\documentclass[11pt]{amsart}
\usepackage[T1]{fontenc}
\usepackage{txfonts}
\usepackage{times}
\usepackage{amssymb,verbatim}
\usepackage[all]{xy}
\usepackage{subfig}
\usepackage{tikz}
\usepackage{color}
\usepackage{a4wide}
\usepackage{hyperref}
\usepackage{wrapfig}
\usetikzlibrary{arrows}

\newcommand{\R}{\mathbb{R}}
\newcommand{\C}{\mathbb{C}}

\newcommand\Z{\mathbb{Z}}

\newcommand{\N}{\mathbb{N}}
\newcommand{\Q}{\mathbb{Q}}
\renewcommand{\H}{\mathcal{H}}
\newcommand{\T}{\mathrm{T}}

\newcommand{\SL}{{\rm SL}}
\newcommand{\PSL}{\mathrm{PSL}}
\newcommand{\GL}{{\rm GL}}

\newcommand{\Aa}{\mathrm{Area}}

\newcommand{\F}{\mathbf{F}}

\newcommand{\ol}{\overline}
\newcommand{\ul}{\underline}

\newcommand{\ra}{\rightarrow}
\newcommand{\id}{\mathrm{id}}
\newcommand{\Id}{\mathrm{Id}}
\newcommand{\vide}{\varnothing}

\newcommand{\inter}{\mathrm{int}}
\newcommand{\Sig}{\Sigma}

\newcommand{\g}{\gamma}
\newcommand{\G}{\Gamma}

\newcommand{\dist}{\mathbf{d}}

\newcommand{\setm}{\setminus}
\newcommand{\leng}{\mathrm{leng}}
\newcommand{\Curv}{\mathrm{Curv}}
\newcommand{\ACurv}{\mathrm{ACurv}}
\newcommand{\diam}{\mathrm{diam}}

\newcommand{\Cc}{\mathcal{C}}
\newcommand{\CC}{\mathbf{C}}

\newcommand{\DD}{\mathbf{D}}

\newcommand{\Fc}{\mathcal F}
\newcommand{\FF}{\mathbf{F}}
\newcommand{\GG}{\mathcal{G}}
\newcommand{\Hh}{\mathbb{H}}

\newcommand{\Ic}{\mathcal{I}}
\newcommand{\Jb}{\mathbf{J}}

\newcommand{\Lc}{\mathcal{L}}
\newcommand{\LL}{\mathbf{L}}

\newcommand{\NN}{\mathbf{N}}
\newcommand{\Nc}{\mathcal{N}}
\newcommand{\Pb}{\mathbb{P}}

\newcommand{\RP}{\mathbb{RP}}
\newcommand{\Rec}{\mathrm{R}}

\newcommand{\Sb}{\mathbf{S}}

\newcommand{\Vc}{\mathcal{V}}

\newcommand{\Wc}{\mathcal{W}}

\newcommand{\Sys}{\varrho}
\newcommand{\ii}{\mathbf{i}}
\newcommand{\Aff}{\mathrm{Aff}}
\newcommand{\Tria}{\mathbb{T}}
\newcommand{\Stab}{\mathrm{Stab}}
\newcommand{\dd}{\mathrm{d}}
\renewcommand{\gg}{\mathbf{g}}

\newtheorem{Theorem}{Theorem}[section]
\newtheorem{Corollary}[Theorem]{Corollary}
\newtheorem{Lemma}[Theorem]{Lemma}
\newtheorem{Proposition}[Theorem]{Proposition}
\newtheorem{Remark}[Theorem]{Remark}
\newtheorem{Definition}[Theorem]{Definition}

\title[Tessellation]{Topological Veech dichotomy and tessellations of the hyperbolic plane}
\author{Duc-Manh Nguyen}
\thanks{The author thanks the VIASM Hanoi for its hospitality during the preparation of this article.}

\address{IMB Bordeaux, CNRS UMR 5251\newline
Universit\'e de Bordeaux \newline
351, Cours de la Lib\'eration \newline
33405 Talence \newline
FRANCE}

\email{duc-manh.nguyen@math.u-bordeaux.fr}
\date{\today}

\begin{document}
\maketitle

\begin{center}
{\it To the memory of William A. Veech}
\end{center}

\begin{abstract}
For {\em every } half-translation surface with marked points $(M,\Sig)$, we construct an associated tessellation $\Pi(M,\Sig)$ of the Poincar\'e upper half plane whose tiles have finitely many sides and area at most $\pi$. 
The tessellation $\Pi(M,\Sig)$ is equivariant with respect to the action of $\PSL(2,\R)$, and invariant with respect to (half-)translation covering. 
In the case $(M,\Sig)$ is the torus $\C/\Z^2$ with a one marked point, $\Pi(\C/\Z^2,\{0\})$ coincides with the {\em iso-Delaunay tessellation} introduced by Veech~\cite{Vee11} (see also \cite{Bow:thesis, Bow10}) as both tessellations give the Farey tessellation.
As application,  we obtain a bound on the volume of the corresponding Teichm\"uller curve in the case $(M,\Sig)$ is  a Veech surface (lattice surface).
Under the assumption that $(M,\Sig)$ satisfies the topological Veech dichotomy, there is a natural graph $\GG$  underlying $\Pi(M,\Sig)$ on which the Veech group $\G$ acts by automorphisms. We show  that $\GG$ has infinite diameter and is Gromov hyperbolic.
Furthermore, the quotient $\ol{\GG}:=\GG/\G$ is a finite graph if and only if $(M,\Sig)$ is actually a Veech surface, 
in which case we provide an algorithm to determine the graph $\ol{\GG}$ explicitly. 
This algorithm also allows one to get a generating family  and a ``coarse" fundamental domain of the Veech group $\G$.
\end{abstract}

\section{Introduction}

\subsection{Embedded triangles and tessellation of the Poincar\'e upper half plane}
 {\em Half-translation surfaces} are flat surfaces defined by meromorphic quadratic differentials with at most simple poles on compact Riemann surfaces.
If the quadratic differential is the square of an Abelian differential (holomorphic one-form) then we have a {\em translation surface}.
Otherwise, there is a canonical (ramified) double covering of the Riemann surface such that the pullback of this quadratic differential is the square of a holomorphic $1$-form, we will call this the orienting double cover. For a thorough introduction to the subject we refer to \cite{MaTa02, Zorich:survey, Lan04}.

Let   $M$ be  a half-translation surface defined by a quadratic differential $(X,q)$.
Let $\Sig$ be a finite subset of $M$ that contains all the conical singularities of the flat metric.
We will call the pair $(M,\Sig)$ a half-translation surface with marked points.
By a slight abuse of notation, we will call $\Sig$ the set of singularities of $M$.

\begin{Definition}\label{def:emb:tri}
An {\em embedded triangle} of $M$ with vertices in $\Sig$, or an embedded triangle in $(M,\Sig)$ for short, is the image of a map $\varphi: \T \ra M$, where $\T$ is a triangle in the plane $\R^2$,  such that
\begin{itemize}
\item[(i)] $\varphi$ maps the vertices $\T^{(0)}$ of $\T$ to $\Sig$,

\item[(ii)] the restriction of $\varphi$   to $\T\setm\T^{(0)}$ is an embedding with image in $M\setm\Sig$, and

\item[(iii)] $\varphi^*q=dz^2$.
\end{itemize}
We denote the set of embedded triangles in $M$ with vertices in $\Sig$ by $\Tria(M,\Sig)$.
\end{Definition}

\begin{Remark}\label{rk:def:emb:tri}
Our definition is slightly different from the definition in \cite{SW:Veech} in that we do not allow a point in the interior of a side of $\T$ to get mapped to a point in $\Sig$.
\end{Remark}
In what follows, we will  sometimes use the same notation for a triangle in $\R^2$ and its image by  a map $\varphi$ as above.

Consider now the canonical orienting cover $\pi: \hat{M} \ra M$, where $\hat{M}$ is a translation surface defined by a holomorphic $1$-form $\hat{\omega}$.
By convention, if $M$ is itself a translation surface then we take $\hat{M}=M$, and $\pi= \id$.
Let $\hat{\Sig}=\pi^{-1}(\Sig)$.
If $\pi$ is a double cover, then the pre-image of a saddle connection $a$ in $(M,\Sig)$ consists of two geodesic segments in $\hat{M}$ with endpoints in $\hat{\Sig}$.
For any directed arc on $\hat{M}$ with endpoints in $\hat{\Sig}$, the integral of $\hat{\omega}$ along this arc is called its {\em period}.
We will call the period of either segment in the pre-image of $a$ its {\em period}.
This is a complex number determined up to sign.
If $\pm(a_x+\imath a_y), \; a_x,a_y \in  \R$, is the period of $a$, we define the slope of $a$ to be
$$
k_a:=\frac{a_x}{a_y} \in \R \cup \{\infty\}.
$$
Let $\Hh$ denote the Poincar\'e upper half plane. Given an embedded triangle $\T$ in  $\Tria(M,\Sig)$, let $k_1, k_2,k_3 \in \R\cup\{\infty\}$ be the slopes of the sides of $\T$. We denote by $\Delta_\T$ the hyperbolic ideal triangle in $\Hh$ whose vertices are $\{k_1,k_2,k_3\}$. Denote by $\Ic(M,\Sig)$ the set of all the ideal triangles arising from elements of $\Tria(M,\Sig)$.
Let $\Cc(M,\Sig)$ denote the set of points $\R\cup\{\infty\}$ that are vertices of ideals triangles in $\Ic(M,\Sig)$, and  $\Lc(M,\Sig)$ the sets of hyperbolic geodesics that are sides of elements of $\Ic(M,\Sig)$.

Recall that a {\em tessellation} of the upper half plane is a family of convex hyperbolic polygon of finite area (but not necessarily compact) that cover $\Hh$ such that two polygons in this family intersect in either a common vertex, or a common side. Elements of this family are called {\em tiles} of the tessellation. Our first result is the following

\begin{Theorem}\label{thm:ideal:tri:tess}
 For any half-translation surface with marked points $(M,\Sig)$, the geodesics in $\Lc(M,\Sig)$ define a tessellation $\Pi(M,\Sig)$ of $\Hh$, each tile of $\Pi(M,\Sig)$ has finitely many sides and area at most $\pi$.  The tessellation $\Pi(M,\Sig)$ is invariant with respect to  half-translation coverings, that is,  if $(M',\Sig')$ is a half-translation covering of $(M,\Sig)$, then $\Pi(M',\Sig')=\Pi(M,\Sig)$.
\end{Theorem}

\begin{Remark}\label{rk:compare:Farey}\hfill
\begin{itemize}
\item[(i)] We refer to Section~\ref{sec:trans:cover} for a detailed discussion on (half-)translation coverings.

\item[(ii)] In an earlier version of this paper, we showed that $\Pi(M,\Sig)$ is a tessellation of $\Hh$ only for the case $(M,\Sig)$ satisfies the topological Veech dichotomy. It turns out that  the same conclusion holds for any $(M,\Sig)$.
\end{itemize}
\end{Remark}

If $M$ is the standard torus $\C/\Z^2$ and $\Sig=\{0\}$, then $\Pi(M,\Sig)$ is the Farey tessellation. Indeed, consider an embedded triangle $\T$ in $(\C/\Z^2, \{0\})$ with the slopes of its sides being $k_i=p_i/q_i,\; p_i,q_i \in \Z, \gcd(p_i,q_i)=1, \; i=1,2,3$. If we cut $M$ along two sides of $\T$, we then get a parallelogram with the third side being a diagonal. Since the area of this parallelogram must be equal to the area of $M$, we have
$$
\left| \det \left(\begin{array}{cc} p_i & p_{i+1} \\ q_i & q_{i+1} \end{array} \right) \right|=|p_iq_{i+1}-p_{i+1}q_i|=1.
$$
for $i=1,2,3$, with the convention $(p_4,q_4)=(p_1,q_1)$. Thus $\Pi(\C/\Z^2,\{0\})$ is the Farey tessellation.
It follows from Theorem~\ref{thm:ideal:tri:tess} that if $(M,\Sig)$ is a translation covering of $(\C/\Z^2,\{0\})$, then $\Pi(M,\Sig)$ is the Farey tessellation as well.

Recall that a square-tiled surface is a translation surface which is obtained by gluing some copies of the unit square.
If $M$ is a square-tiled surface, and $\Sig$ is the set of singularities of the flat metric on $M$, then $(M,\Sig)$ is not necessarily a translation covering of $(\C/\Z^2,\{0\})$. 
This is because the natural map from $M$ onto $\C/\Z^2$ may send some regular points in $M$ to $0$.  
Thus, in this case  $\Pi(M,\Sig)$ is not necessarily the Farey tessellation.
In Figure~\ref{fig:tess:D16:H2}, we give the tessellation associated to a square-tiled surface $M$ in $\H(2)$ which is composed by $4$ unit squares, where $\Sig$ consists of the unique singularity of $M$.
 \begin{figure}[htb]
 \centering
\includegraphics[width=6cm]{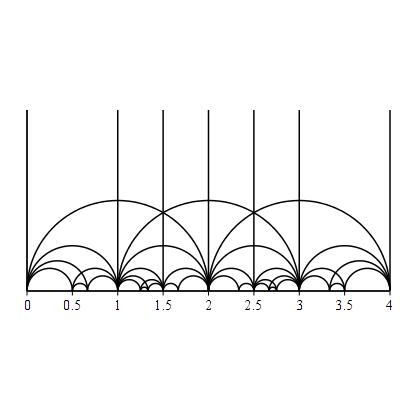}
 \caption{Tessellation associated with a square-tiled surface in $\mathcal{H}(2)$, in the horizontal direction we have a single cylinder composed by $4$ squares}
  \label{fig:tess:D16:H2}
 \end{figure}


In \cite{Vee11}, Veech considered a tessellation of $\Hh$ associated to $(M,\Sig)$ which arises from the Delaunay partition of the surfaces in the orbit $\SL(2,\R)\cdot(M,\Sig)$.
To define this tessellation, one first identifies $\Hh$ with the quotient $\SL(2,\R)/{\mathrm{SO}(2,\R)}$.
Each tile $P$ of this tessellation corresponds to a subset $\tilde{P}$ of $\SL(2,\R)$ such that the Delaunay partition of $A\cdot (M,\Sig)$ remains the same as $A$ varies in $\tilde{P}$.
Each side of $P$ is a geodesic segment $\eta$ specified by the following condition: there are two embedded triangles $\T,\T'$ in $\Tria(M,\Sig)$ sharing a common side such that for any $A$ in $\tilde{\eta}$ (here $\tilde{\eta}$ is the pre-image of $\eta$ in $\SL(2,\R)$), the quadrilateral $A\cdot(\T\cup\T')$ is inscribable in a circle  (see \cite{Bow:thesis} for some interesting characterizations of this tessellation).
It turns out that every tile $P$ has at most $6g(M)-6+3|\Sig|$ sides, where $g(M)$ is the genus of $M$, and $|\Sig|$  is the cardinality of $\Sig$.
However, there is no known bound on the area of $P$. In~\cite{Bow:thesis}, it has been conjectured that the area of $P$ is at most $\pi$.

We will refer to the tessellation described above as the {\em iso-Delaunay tessellation} associated to $(M,\Sig)$ (this terminology was introduced in \cite{Bow10}). Iso-Delaunay tessellations can also be seen as the decomposition of the Teichm\"uller disc generated by $(M,\Sig)$ which is induced by a cell decomposition of the Teichm\"uller space. Note that iso-Delaunay tessellations are also invariant by half-translation coverings (see \cite[Prop.4.1]{Vee11})

Surprisingly, even though the origins of the iso-Delaunay tessellation and of the tessellation $\Pi(M,\Sig)$ seem to be unrelated, these two tessellations do coincide when $(M,\Sig)$ is the standard torus with one marked point $(\C/\Z^2,\{0\})$, in which case they both give the Farey tessellation.
In \cite[Th.1.3]{Vee11}, Veech showed that the $\GL(2,\R)$-orbit  (of the orienting double cover) of $(M,\Sig)$ contains a translation cover of the standard torus $(\C/\Z^2, \{0\})$ if and only if the associated iso-Delaunay tessellation is isomorphic to the Farey tessellation.
Inspired by this result, we will show 
\begin{Theorem}\label{th:sq:tiled:Farey:tess}
The tessellation $\Pi(M,\Sig)$ is isomorphic to  the Farey tessellation if and only if $\GL(2,\R)\cdot(\hat{M},\hat{\Sig})$ contains a translation cover of $(\C/\Z^2,\{0\})$, where $(\hat{M},\hat{\Sig})$ is the orienting double cover of $(M,\Sig)$.
\end{Theorem}

\begin{Remark}\label{rmk:compare:Veech:tess}
It would be interesting to determine how the iso-Delaunay tessellation and the tessellation in Theorem~\ref{thm:ideal:tri:tess} are related in general.
We hope to be able to address this question in future work.
It is also worth noticing that other tessellations of $\Hh$ associated to half-translation surfaces have been introduced by Smillie-Weiss~\cite{SW:Veech} (see also~\cite{Mu13} for related constructions).
\end{Remark}

\subsection{Action of the Veech group and a bound on the volume of Teichm\"uller curves}
An {\em affine automorphism} of $(M,\Sig)$ is an orientation preserving homeomorphism $f :M \ra M$ such that $f(\Sig)=\Sig$, and there is a matrix $A \in \SL(2,\R)$ such that on the local charts given by the flat metric on $M\setminus \Sig$, $f$ has the form $v \mapsto \pm A\cdot v+ c$, where $c \in \R^2$ is constant. The group of affine automorphisms of $(M,\Sig)$ will be denoted by $\Aff^+(M,\Sig)$.
To each element of $\Aff^+(M,\Sig)$, we have a corresponding element of $\PSL(2,\R)$ by the derivative mapping $D: \Aff^+(M,\Sig) \ra \PSL(2,\R)$.
The image of $\Aff^+(M,\Sig)$ under $D$ is called the {\em Veech group} of $(M,\Sig)$ and denoted by $\G(M,\Sig)$.
The pair $(M,\Sig)$ is called a {\em Veech surface} (or equivalently a {\em lattice surface}) if $\G(M,\Sig)$ is a lattice in $\PSL(2,\R)$.

Note that by construction, $\G(M,\Sig)$ permutes the elements of $\Tria(M,\Sig)$, hence acts naturally on the sets $\Ic(M,\Sig),\Lc(M,\Sig)$, and $\Cc(M,\Sig)$.
Denote by $\allowbreak \ol{\Ic}(M,\Sig), \; \ol{\Lc}(M,\Sig), \; \ol{\Cc}(M,\Sig)$ the respective quotients.
For simplicity, in what follows, once the pair $(M,\Sig)$ is specified, we will drop it from the notation of the associated objects, {\em i.e.} $\Ic, \Cc, \Lc,\G$, etc...

Understanding Veech groups is a central problem in Teichm\"uller dynamics. Various aspects of this problem has been addressed by several authors,
see for instance \cite{HS01, HS:inf:Veech, HubLan06, KS00, McM_hilbert, McM_flux, SW10, SW:Veech, Bow10, Sch04, WSch15, Mu13}.
One of the main goals of this paper is to contribute to the investigation of Veech groups by using the properties of the flat metric.
In particular, as a consequence of Theorem~\ref{thm:ideal:tri:tess}, we get a bound on the volume of the Teichm\"uller curve.

\begin{Theorem}\label{thm:area:TC:bound}
The surface $(M,\Sig)$ is a Veech surface if and only if  $\ol{\Ic}$ is a finite set and we have
 \begin{equation}
\label{eq:lattice:vol:bound}
 \Aa(\Hh/\G(M,\Sig)) \leq \pi\cdot \# \ol{\Ic}.
\end{equation}
\end{Theorem}

\begin{Remark}\label{rm:bound:area}
In Section~\ref{sec:quot:graph:diam}, we will introduce an algorithm to determine the cardinality of $\ol{\Ic}$ in the case $(M,\Sig)$ is a  Veech surface.
\end{Remark}


\subsection{The graph of periodic directions}\label{sec:per:dir:graph}
We say that a half-translation surface satisfies the {\em topological Veech dichotomy} if for every $\theta \in \RP^1$, the foliation $\xi_\theta$ of $M$ by straight lines in direction $\theta$ is either periodic, or minimal. This property can also be stated as follows: if there is a saddle connection in direction $\theta$ then the surface is decomposed into cylinders in this direction.
Any Veech surface satisfies this dichotomy.
However, is shown in \cite{CHM08, HS:inf:Veech,LanNg:cp} that there exist surfaces that satisfy the topological Veech dichotomy without being Veech surfaces (see also \cite{SW:Veech}).

Assume from now on that $(M,\Sig)$ satisfies the topological Veech dichotomy.
We will construct a graph $\GG$ underlying the tessellation $\Pi$ which comes equipped with a natural action of $\G$.
One may expect that the quotient graph $\GG/\G$ captures some geometric properties of the Teichm\"uller curve $\Hh/\G$.
The graph $\GG$ is defined as follows:
\begin{itemize}
 \item[$\bullet$] the vertex set $\GG^{(0)}$ of $\GG$ is $\Cc\sqcup \Ic$,

 \item[$\bullet$] for any pair $(k,\Delta) \in \Cc\times\Ic$, there is an edge  connecting $k$ and $\Delta$ if and only if $k$ is a vertex of $\Delta$,

 \item[$\bullet$] there is no edge between two elements of $\Cc$ nor two elements of $\Ic$.
\end{itemize}
We set the length of every edge of $\GG$ to be $1/2$.
We will call $\GG$ the {\em graph of periodic directions} of $(M,\Sig)$.

By construction, we have a natural action of $\G$ on $\GG$ by automorphisms.
We denote by $\ol{\GG}$ the quotient of $\GG$ by $\G$.
In the perspective of understanding the Veech group, we study of the geometry of $\GG$.
In particular, we will show

\begin{Theorem}\label{thm:G:per:dir:prop}
Let $(M,\Sig)$ be a half-translation surface satisfying the topological Veech dichotomy. Then the graph $\GG$ of periodic directions of $(M,\Sig)$ is connected, has infinite diameter, and is Gromov hyperbolic. The Veech group $\G$ acts freely on the set of edges of $\GG$. Moreover, $(M,\Sig)$ is a Veech surface if and only if $\ol{\GG}$ is a finite graph.
\end{Theorem}

\noindent {\bf Example:} in the case $(M,\Sig)=(\C/\Z^2,\{0\})$, we have $\G(\C/\Z^2,\{0\})=\PSL(2,\Z)$, and each of $\ol{\Ic}(\C/\Z^2,\{0\})$ and $\ol{\Cc}(\C/\Z^2,\{0\})$ contains a single element.
Let $\Delta_0$ be the hyperbolic ideal triangle whose vertices are $\{0,1,\infty\}$. Since $\PSL(2,\Z)$ contains an element that fixes $\Delta_0$ and permutes its vertices cyclically, namely $\pm \left(\begin{smallmatrix} 1 & -1 \\ 1 & 0 \end{smallmatrix}\right)$, we deduce that the graph $\ol{\GG}(\C/\Z^2, \{0\})$ consists of a unique segment joining the unique element of $\ol{\Ic}(\C/\Z^2,\{0\})$ and the unique element of $\ol{\Cc}(\C/\Z^2,\{0\})$.

Since $\GG$ is invariant with respect to (half)-translation coverings and the group $\G$ acts freely on the set of edges of $\GG$,  we get

\begin{Corollary}\label{cor:sq:tiled:index:V:Gp}
If $(M,\Sig)$ is a translation cover of $(\C/\Z^2,\{0\})$, then the number of edges of $\ol{\GG}$ is equals to $[\PSL(2,\Z):\G]$.
\end{Corollary}

Generally, it would be interesting to determine to what extend the geometry and topology of the hyperbolic surface $\Hh/\G$ is encoded in $\ol{\GG}$.
We  hope to return to this problem in a near future.

\subsection{Fundamental domain and generators of the Veech group}
%
In the appendix, we propose an algorithm to determine the graph $\ol{\GG}$ explicitly in the case $(M,\Sig)$ is a Veech surface, and hence to get a bound for $\Aa(\Hh/\G)$ in this case using Theorem~\ref{thm:area:TC:bound}. This algorithm also allows us to calculate a generating set  and a ``coarse'' fundamental domain of the Veech group $\G$.

\subsection*{Acknowledgment:} The author is grateful to Howie Masur for the helpful conversation. He warmly thanks Vincent Delecroix, Samuel Leli\`evre, and  Huiping Pan for the useful comments on an earlier version of this paper. He thanks the referee for the valuable comments, especially on Veech's work~\cite{Vee11}, which help to improve some aspects of the paper.

\section{Embedded triangles and coverings}\label{sec:trans:cover}
\subsection{Half-translation covering}\label{subsec:tree:trans:cover}
Let $(M',\Sig')$ and $(M,\Sig)$  be two  half-translation surfaces with marked points. Assume that $M'$ and $M$ are defined by two pairs (Riemann surface, quadratic differential) $(X',q')$ and $(X,q)$ respectively. Let $\Sig$ (resp. $\Sig'$) be a finite subset of $M$ (resp. of $M'$) that contains all the zeros and (simple) poles of $q$ (resp. of $q'$).
A {\em half-translation covering} is a ramified covering of Riemann surfaces $f: X' \ra X$ which is branched over $\Sig$ such that $\Sig'=f^{-1}(\Sig)$ and  $q'=f^*q$.
In particular, an orienting double covering map is a half-translation covering.
If both $M$ and $M'$ are translation surfaces then such a map is called a {\em translation covering}.
Note that such coverings are also known as {\em balanced coverings} (see e.g. \cite{HS01}).


Two half-translation surfaces with marked points are said to be {\em affine  equivalent} if they belong to the same $\PSL(2,\R)$-orbit up to scaling.
As suggested by Hubert and Schmidt~\cite{HS01}, we can define the  notion of {\em tree of half-translation coverings}  as follows:  such a tree is a connected  acyclic directed graph whose vertices are equivalence classes of half-translation surfaces with marked points, and two vertices, represented by $(M_1,\Sig_1)$ and $(M_2,\Sig_2)$,
are connected by a directed edge from first to the second if there exists a half-translation covering map $f:(M_1,\Sig_1) \ra (M'_2,\Sig'_2)$,
where $(M'_2,\Sig'_2)$ is a surface in the equivalence class of $(M_2,\Sig_2)$.
Note that  any loop formed by oriented edges of this graph must be trivial (constant).
It follows from a result by M\"oller~\cite{Mo_invent06} that if a tree of translation coverings contains a surface which is not a torus cover then it has a root.


\subsection{Invariance of the set of embedded triangles}
Throughout this section  $(M,\Sig)$ will be a fixed half-translation surface with marked points, which is defined by a meromorphic quadratic differential $(X,q)$ whose poles are all simple.
\begin{Lemma}\label{lm:loc:isom:embed}
Let $\varphi: \T \ra M$ be a map from an Euclidean  triangle $\T$ to $M$ such that
\begin{itemize}
  \item the vertices of $\T$ are mapped to points in $\Sig$,
  \item $\varphi(\T\setm\T^{(0)}) \subset M\setm\Sig$, where $\T^{(0)}$ is the set of vertices of $\T$,
  \item the restriction of $\varphi$ to $\T\setm\T^{(0)}$ is locally isometric.
\end{itemize}
Then the restriction of $\varphi$ to $\T\setm\T^{(0)}$ is an embedding.
\end{Lemma}
\begin{proof}
In what follows we identify $\T$ with a subset of $\R^2$.
Assume that there are two points $x_1,x_2 \in \T\setm\T^{(0)}$ such that $\varphi(x_1)=\varphi(x_2)$. Since a triangle is a convex subset of the plane, the segment $\ol{x_1x_2}$ is contained in $\T$. Its image by $\varphi$ is a loop $\gamma$ in $M\setm\Sig$. Let $h \in \{\pm\Id\}\ltimes \R^2$ be the holonomy of $\gamma$.

If $h(v)=-v+c$, then we have $x_2=-x_1+c$. Let $x_0$ be the midpoint of $\ol{x_1x_2}$ then $x_0=c/2$. Thus $h(x_0)=x_0$, which means that $x_0$ is mapped to singular point with cone angle $\pi$ of  $M$.
By assumption $\Sig$ contains all the singularities of $M$. Since $x_0\in \T\setm\T^{(0)}$ we have a contradiction to the assumption that $\varphi(\T\setm\T^{(0)})\subset M \setm\Sig$. Hence this case does not occur.

If $h(v)=v+c$ then $c=\overrightarrow{x_1x_2}$.
Let $\T_c$ denote triangle $\T+c$. Let $\T^*$ (resp. $\T^*_c$) denote the triangles $\T$ (resp. $\T_c$) with its vertices removed.
By assumption, we have $\T^*\cap \T^*_c \neq \varnothing$. One  readily checks that this condition implies that either $\T^*_c$ contains a vertex of $\T$, or $\T^*$ contains a vertex of $\T_c$.
It follows that there is a vertex $v_0$ of $\T$ such that either $v_0+c \in \T^*$ or $v_0-c \in \T^*$.
Since $v_0$ and $v_0\pm c$ are mapped to the same point in $M$, we have again a contradiction to the assumption that $\varphi(\T\setm\T^{(0)})\subset M\setm\Sig$.
Thus the restriction of $\varphi$ to $\T\setm\T^{(0)}$ is an embedding.
\end{proof}

\begin{Lemma}\label{lm:blced:cov:same:tria}
  Let $f: (M',\Sig') \ra (M,\Sig)$ be a half-translation covering of half-translation surfaces with marked points.
  Then we have $\Tria(M',\Sig')=\Tria(M,\Sig)$. 
\end{Lemma}
\begin{proof}
Consider an embedded triangle $\varphi :\T \ra M'$ with $\varphi(\T^{(0)}) \subset \Sig'$. Composing with $f$, we get a map $\phi:=f\circ \varphi: \T \ra M$ with $\phi(\T^{(0)}) \subset \Sig$. Since $f$ is a half-translation covering, the map $\phi$ is a local isometry on $\T\setm\T^{(0)}$ and satisfies
$\phi^*q=dz^2$ in the interior of $\T$ (where $q$ is the quadratic differential defining the flat metric of $M$).
It follows from Lemma~\ref{lm:loc:isom:embed} that $f\circ\varphi(\T)$ is an embedded triangle in $(M,\Sig)$.

On the other hand, given an embedded triangle $\phi :\T \ra M$ with vertices in $\Sig$, we can lift $\phi$ to a map $\hat{\phi}: \T \ra M'$ which is also a local isometry on $\T \setm\T^{(0)}$. By Lemma~\ref{lm:loc:isom:embed}, we have that $\hat{\phi}: \T \ra M'$ is also an embedded triangle in $M'$ with vertices in $\Sig'$. Thus  the sets $\Tria(M',\Sig')$ and $\Tria(M,\Sig)$ are equal.
\end{proof}



\section{Tessellation of the hyperbolic plane}\label{sec:tessellation}
Our goal now is to give the proofs of Theorem~\ref{thm:ideal:tri:tess}  and Theorem~\ref{thm:area:TC:bound}.
By Lemma~\ref{lm:blced:cov:same:tria}, we can replace $(M,\Sig)$ by its orienting cover.
Therefore, throughout this section  $(M,\Sig)$ is a translation surface with marked points. 
We also normalize (using the action of $\GL^+(2,\R)$ such that $\Aa(M)=1$.

\subsection{The ideal triangles in $\Ic(M,\Sig)$ cover $\Hh$}
We first show

\begin{Lemma}\label{lm:ideal:tri:cover}
Let $z$ be a point in $\Hh$. Then either there is an embedded triangle $\T \in \Tria(M,\Sig)$ such that $z$ is contained in the interior of the ideal triangle $\Delta_\T$ associated with $\T$, or there exist two embedded triangles $\T_1,\T_2 \in \Tria(M,\Sig)$ such that $\Delta_{\T_1}\cup\Delta_{\T_2}$ is an ideal quadrilateral that contains $z$ in one of its diagonals.
\end{Lemma}
\begin{proof}
Let $A$ be a matrix in $\SL(2,\R)$ such that $z=A^{-1}(\imath)$. Consider the surface $(M',\Sig'):=A\cdot(M,\Sig)$.  Let $\Sys(M',\Sig')$ denote the length of the shortest saddle connection on $M'$ (with endpoints in $\Sig'$). Let $s_0$ be a saddle connection  in $M'$ such that $|s_0|=\Sys(M',\Sig')$. Replacing $A$ by $RA$, where $R\in \mathrm{SO}(2,\R)$, if necessary (note that $(RA)^{-1}(\imath)=A^{-1}(\imath)=z$), we can assume that $s_0$ is horizontal.

Consider the vertical separatrices of $(M',\Sig')$, that is the vertical geodesic rays emanating from the points in $\Sig'$. We have two cases

\begin{itemize}
\item \ul{Case (a):} $\inter(s_0)$ intersects some vertical separatrices. For each vertical separatrix intersecting $\inter(s_0)$, consider the subsegment from its origin to its first intersection with $\inter(s_0)$. Pick  a  segment of minimal length $u_0$ in this family. Then there is an embedded triangle $\T'$  containing this vertical segment which is bordered by $s_0$ and two other saddle connections denoted by $s_1,s_2$. This triangle can be constructed as follows:  one can identify $s_0$ with a horizontal segment and $u_0$ with a vertical segment in the plane. Let $P_1,P_2$ denote the endpoints of the segment corresponding to $s_0$, and $P_0,Q_0$ denote the  endpoints of the segment corresponding to $u_0$, where $Q_0 \in \ol{P_1P_2}$. Let $\T'$ be the triangle with vertices $P_0,P_1,P_2$. Since  the  length of $u_0$ is minimal among the vertical segments from a point in $\Sig'$ to a point in $\inter(s_0)$, the developing map induces a map $\varphi: \T' \ra M'$ which is locally isometric.
By Lemma~\ref{lm:loc:isom:embed}, the image of $\varphi$ is an embedded triangle in $(M',\Sig')$. By construction $s_i=\varphi(\ol{P_0P_i}), \; i=1,2$.

Let $k_1,k_2$ be the slopes of $s_1$ and $s_2$ respectively. Note that we always have $k_1k_2<0$. Without loss of generality, we can assume that $k_1 > 0 > k_2$ (equivalently, $P_1$ is the left endpoint of $s_0$). We now claim that
\begin{equation}\label{eq:bound:k2-k3}
k_1-k_2 \leq \frac{2}{\sqrt{3}}.
\end{equation}
Let $x_i$ be the length of the segment $\ol{P_iQ_0}, \; i=1,2$, and $y$ be the length of the segment $\ol{P_0Q_0}$.
By definition, we have $k_1=x_1/y$ and $k_2=-x_2/y$. Hence
$$
k_1-k_2 =\frac{x_1+x_2}{y}=\frac{x}{y}
$$
where $x=x_1+x_2=|s_0|$.
By definition, we have $|s_0| \leq \min \{|s_1|,|s_2|\}$, therefore $x^2 \leq \min\{x_1^2+y^2, x_2^2+y^2\}$. But since $x=x_1+x_2$, we have $\min\{x_1,x_2\} \leq \frac{x}{2}$. Thus we have
$$
x^2 \leq \frac{x^2}{4} +y^2 \text{ which implies } \frac{x}{y} \leq \frac{2}{\sqrt{3}}
$$
which proves the claim.

Since we have $k_1 > 0 > k_2$ and $k_1-k_2 \leq 2/\sqrt{3} $, the radius of the half circle perpendicular to the real axis passing through $k_1$ and $k_2$ is at most $ \frac{1}{\sqrt{3}}<1$. Thus it  cannot separate $\imath$ and $\infty$. It follows in particular that $\imath$ is contained in the ideal triangle $\Delta_{\T'}$ with vertices $\{\infty, k_1,k_2\}$.

Since $(M,\Sig)=A^{-1}\cdot (M',\Sig')$, $\T=A^{-1}(\T')$ is an embedded triangle in $\Tria(M,\Sig)$. Note that the slope of the sides of $\T$ are $\{A^{-1}(\infty),A^{-1}(k_1),A^{-1}(k_2)\}$ (here we consider the usual action of $A^{-1}$ on $\R\cup\{\infty\}=\partial\Hh$) which means that $\{A^{-1}(\infty),A^{-1}(k_1),A^{-1}(k_2)\}$ are the vertices of the ideal triangle $\Delta_{\T}$, or equivalently $\Delta_\T=A^{-1}(\Delta_{\T'})$. Since $\imath$ is contained in the interior of $\Delta_{\T'}$ and $z=A^{-1}(\imath)$ by definition, we conclude that $z$ is contained in the interior of $\Delta_{\T}$.\\

\item \ul{Case (b):}  no-vertical saddle connection intersects $\inter(s_0)$. In this case $s_0$ is contained in a vertical cylinder $C$ of $(M',\Sig')$. We can realize the cylinder $C$ as the image of a rectangle $\Rec$ in the plane under a locally isometric mapping $\varphi: \Rec \ra M$ such that the restriction of $\varphi$  to $\inter(\Rec)$ is injective, and $\varphi$ maps both the bottom and top sides of $\Rec$ onto $s_0$.

Let $P_1$ and $P_2$ denote the left and right endpoints of the bottom side of $\Rec$ respectively.
There is a subsegment of the left side of $\Rec$, with $P_1$ being an endpoint, that is mapped to a vertical saddle connection $r_1$ in the boundary of $C$.
Similarly, there is a subsegment of the right side of $\Rec$, with $P_2$ being an endpoint, that is mapped to a vertical saddle connection $r_2$ in the boundary of $C$.
Let $P'_1$ and $P'_2$ denote the upper endpoints of $r_1$ and $r_2$ respectively.
Let $s_1$ (resp. $s_2$) denote the saddle connection that is the image of $\ol{P'_1P_2}$ (resp. of $\ol{P_1P'_2}$) under $\varphi$.
Remark that  $s_0,s_i,r_i$  bound an embedded triangle $\T'_i$, which is entirely contained in $C$, for $i=1,2$.

Let $k_i$ be the slope of $s_i$, then $k_1< 0 < k_2<0$. The vertices of the hyperbolic ideal triangle $\Delta_{\T'_i}$ are $\{\infty, 0,k_i\}$.
Since $k_1k_2 <0$, the vertical line from $\infty$ to $0$ is the common side of the ideal triangles $\Delta_{\T'_1}$ and $\Delta_{\T'_2}$.
Hence $\imath$ is contained in the interior of the ideal quadrilateral formed by  $\Delta_{\T'_1}$ and $\Delta_{\T'_2}$.
By the same arguments as the previous case, we see that there exists two embedded triangles $\T_1,\T_2$ in $\Tria(M,\Sig)$ such that $z$ is contained in a diagonal of the quadrilateral formed by $\Delta_{\T_1}$ and $\Delta_{\T_2}$.
\end{itemize}
\end{proof}

\subsection{Locally finite property of $\Lc$}
\begin{Lemma}\label{lm:loc:finite:geod}
 Let $K$ be a compact subset of $\Hh$. Then the set $\{\gamma \in \Lc, \; \gamma\cap K \neq \varnothing\}$ is finite.
\end{Lemma}
\begin{proof}
Assume that there is an infinite family of geodesics $\{\gamma_n\}_{n\in \N} \subset \Lc$ such that $\gamma_n\cap K\neq \varnothing$.
For each $n\in \N$, let $p_n,q_n \in \partial \Hh$ denote the endpoints of $\gamma_n$.
Since $\partial \Hh \simeq \mathbb{RP}^1$ is compact, by extracting a subsequence if necessary, we can suppose that $\lim_{n\to +\infty}p_n=p$ and  $\lim_{n\to -\infty}q_n=q$.

We first notice that $p\neq q$, this is because if $p=q$, then for any compact $K$, $\gamma_n\cap K=\varnothing$ when $n$ is large  enough.
Let $(p',q')$ be a pair of points in $\partial\Hh$ such that the geodesic joining $p'$ and $q'$ separates $p$ from $q$.
Using the action of $\PSL(2,\R)$, we can further assume that $p'=\infty$ and $q'=0$.
Without loss of generality, we can assume that $q < 0 < p$, and that all the geodesics in the family $\{\gamma_n\}_{n\in\N}$ cross the vertical half-line $\imath\R^+$.
By definition, each geodesic $\gamma_n$ is associated with an embedded triangle $\T_n$  in  $\Tria(M,\Sig)$.
This triangles has two sides $a_n, b_n$ such that the slope of $a_n$ is $p_n$ and the slope of $b_n$ is $q_n$.
Since $\gamma_n$ crosses $\imath\R^+$, we must have $q_n < 0 < p_n$.

Since the slope of $a_n$ belongs to $(0;+\infty)$, we can suppose that its period is $x_n+\imath y_n$, where both $x_n,y_n$ are positive real numbers.
Similarly, since the slope of $b_n$ belongs to $(-\infty;0)$, we can suppose that its slope is $-x'_n+\imath y'_n$, where $x'_n,y'_n$ are both positive.
As a consequence
$$
\Aa(\T_n)=\frac{1}{2}\left|\det\left(\begin{array}{cc} x_n & -x'_n \\ y_n & y'_n \end{array} \right)\right|=\frac{1}{2}(x_ny'_n+y_nx'_n).
$$
Since $\T_n$ is an embedded triangle, we must have $\Aa(T_n)< \Aa(M)=1$. Hence
\begin{equation}\label{eq:area:bound}
 x_ny'_n+y_nx'_n < 2.
\end{equation}
By definition $p_n=x_n/y_n$ and $q_n=-x'_n/y'_n$. Therefore
\begin{equation}\label{eq:slopes:diff}
 p_n-q_n=\frac{x_n}{y_n}+\frac{x'_n}{y'_n}=\frac{x_ny'_n+y_nx'_n}{y_ny'_n} < \frac{2}{y_ny'_n}.
\end{equation}
Since $p_n$ and $q_n$ converge to $p$ and $q$ respectively, there exists $\alpha >0$ such that $p_n-q_n >\alpha$ for all $n\in \N$. Thus \eqref{eq:slopes:diff} implies
\begin{equation}\label{eq:period:bound:1}
 y_ny'_n < \frac{2}{\alpha}
\end{equation}
We have two possibilities:
\begin{itemize}
 \item[$\bullet$] Case 1: both $y_n$ and $y'_n$ are bounded below by a constant $\beta>0$. Then it follows from \eqref{eq:period:bound:1} that both $y_n$ and $y'_n$ are bounded above by $\frac{2}{\alpha\beta}$.
 The inequality \eqref{eq:area:bound} implies that both $x_n$ and $x'_n$ are also bounded above by $\frac{2}{\beta}$. Thus the lengths of $a_n$ and $b_n$ are bounded by some constant $R$.  But it is a well known fact that there are only finitely many saddle connections on $(M,\Sig)$ that have length at most $R$. This contradiction shows that this case cannot occur.

 \item[$\bullet$] Case 2: either $\liminf_{n\to +\infty}y_n=0$, or $\liminf_{n\to \infty}y'_n=0$.
 Let us suppose that  $\liminf_{n\to +\infty}y_n=0$, the case $\liminf_{n\to \infty}y'_n=0$ follows from the same argument.
 By extracting a subsequence, we can assume that $\lim_{n\to \infty} y_n=0$.
 Since
 $$p=\lim_{n\to \infty}p_n=\lim_{n\to \infty}\frac{x_n}{y_n},$$
 it follows that $\lim_{n\to \infty}x_n=0$.
 In particular, $\{a_n\}_{n\in \N}$ is a  sequence of saddle connections that have lengths tend to $0$.
 Since such a sequence cannot exist, we get again a contradiction which proves the lemma.
\end{itemize}
\end{proof}

\subsection{Proof of Theorem~\ref{thm:ideal:tri:tess}}
\begin{proof}
Let $\hat{\Lc}$ denote the union of all the geodesics in $\Lc$.
It follows from Lemma~\ref{lm:loc:finite:geod} that $\hat{\Lc}$ is a closed subset of $\Hh$.
We need to show that every component of the set $\Hh\setminus\hat{\Lc}$ is a hyperbolic polygon with finitely many sides and area at most $\pi$.

Let $z$ be a point in $\Hh\setminus\hat{\Lc}$, and  $P_z$ the component of $\Hh\setminus\hat{\Lc}$ that contains $z$.
By Lemma~\ref{lm:ideal:tri:cover}, $z$ is contained in the interior of an ideal triangle $\Delta$ in $\Ic$.
Let $\Fc$ denote the set of geodesics in  $\Lc$ that intersect $\inter(\Delta)$.
Then $P_z$ is cut out by the boundary of $\Delta$ and the geodesics in $\Fc$.
In particular, $P_z$ is contained in $\Delta$. Therefore, we have $\Aa(P_z) \leq \Aa(\Delta)=\pi$.
It remains to show that $P_z$ has finitely many sides.

Let $p_1,p_2,p_3$ denote the vertices of $\Delta$. Let $V_i$ denote the intersection of $\Delta$ with a horoball tangent to $\partial\Hh$ at $p_i$.
We choose the horoballs such that $V_1,V_2,V_3$ are pairwise disjoint, and $z \not\in V_1\cup V_2 \cup V_3$.

Let $K:=\Delta\setminus\big(V_1\cup V_2 \cup V_3\big)$.
We split $\Fc$ into two subsets $\Fc'$ and $\Fc''$, where $\Fc'$ is the set of geodesics in $\Lc$ that intersect $K$.
Since $K$ is compact, it follows from Lemma~\ref{lm:loc:finite:geod} that $\Fc'$ is a finite set.
Let $P'_z$ denote the domain containing $z$ which is cut out by the sides of $\Delta$ and the geodesics in $\Fc'$.
Since the set $\Fc'$ is finite, $P'_z$ is a convex polygon with finitely many sides.

By definition, $\Fc''$ is the set of geodesics in $\Lc$ that intersect $\Delta$ but disjoint from $K$.
Denote by $\Fc''_i$ the set of geodesics in $\Fc''$ which intersect $V_i$.
Note that if $\gamma$ is a geodesic in $\Fc''_i$, then since $\gamma$ does not cross $K$, it does not intersect $V_j$ if $j\neq i$.
Therefore, $\gamma$ separates $p_i$ from $z$, and we have $\Fc'':=\Fc''_1\sqcup\Fc''_2\sqcup\Fc''_3$.

For $i=1,2,3$, if $\Fc''_i \neq \varnothing$ then we pick a geodesic $\gamma''_i \in \Fc''_i$.
Let $P''_z$ denote the component of $P'_z$ containing $z$ which is cut out by the geodesics $\gamma''_i$.
By construction, $P''_z$ is also a polygon with finitely many sides.

\begin{itemize}
\item[$\bullet$] If none of  $\{p_1,p_2,p_3\}$ is a vertex of $P''_z$, then $P''_z$ is a compact.
Since $P_z$ is cut out by the geodesics in $\Fc''$ that intersect $P''_z$, we conclude by Lemma~\ref{lm:loc:finite:geod}.

\item[$\bullet$] For $i=1,2,3$, if $p_i$ is a vertex of $P''_z$ then $\Fc''_i=\varnothing$, because otherwise we have a geodesic that separates $z$ from $p_i$.
In this case  we remove from $P''_z$ the intersection $P''_z\cap V_i$.
The remaining subset of $P''_z$ is denoted by $Q''_z$. Note that $Q''_z$ is compact.

We now claim that if $\gamma$ is a geodesic in $\Fc''$ which intersects $P''_z$, then $\gamma$ must intersect $Q''_z$.
Indeed, if $\gamma$ intersects $P''_z$ but not $Q''_z$ then it must intersect one of the set $P''_z\cap V_i$, where $\Fc''_i=\varnothing$.
But by definition, $\Fc''_i$ contains $\gamma$. Thus we have a contradiction.

The claim implies that $P_z$ is cut out from $P''_z$ by finitely many geodesics in $\Fc''$ (by Lemma~\ref{lm:loc:finite:geod}).
Therefore, $P_z$ has finitely many sides.
\end{itemize}
The invariance of $\Pi(M,\Sig)$ with respect to half-translation coverings follows from Lemma~\ref{lm:blced:cov:same:tria}.
\end{proof}
\subsection{Proof of Theorem~\ref{th:sq:tiled:Farey:tess}}\label{sec:prf:th:sq:tiled:Farey}
\begin{proof}
Since the tessellation associated with $(\C/\Z^2,\{0\})$ is the Farey tessellation, it follows from Lemma~\ref{lm:blced:cov:same:tria} that if $(M,\Sig)$ is a translation covering of $(\C/\Z^2,\{0\})$, then $\Pi(M,\Sig)$ is the Farey tessellation as well. 
Thus, we only need to show that if $\Pi(M,\Sig)$ is the Farey tessellation, then up to a rescaling the canonical orienting covering of $(M,\Sig)$ is a translation covering of $(\C/\Z^2,\{0\})$.

To simplify the arguments, we can assume that $M$ is a translation surface.
Recall that  every geodesic in the Farey tessellation joins a pair of points $(p/q,p'/q')\in (\Pb_\Q^1)^2$ such that $|pq'-qp'|=1$.

Consider an embedded triangle $\T$ in $\Tria$. Since the slopes of the sides of $\T$ are in $\Q\cup\{\infty\}$, there is an element $A$ of $\SL(2,\Z)$ such that the slopes of the triangle $A\cdot\T$ are $(0,1,\infty)$.
Replacing $(M,\Sig)$ by $A\cdot(M,\Sig)$, we can then assume that the ideal triangle $\Delta$ associated to $\T$ has vertices $\{0,1,\infty\}$.
Recall that the slopes of the horizontal and vertical directions are respectively $\infty$ and $0$.
Denote by $s_1,s_2,s_3$ the sides of $\T$, where  $s_1$ is vertical, the slope of $s_2$ is $1$, and $s_3$ is horizontal (see Figure~\ref{fig:sq:tiled:Farey:tess}).
Rescaling $M$ using $\R_+^*$, we can further assume that the lengths of $s_1$ and $s_3$ are both equal to $1$.
We will show that in this case $(M,\Sig)$ is a translation cover of the torus $(\C/\Z^2, \{0\})$.

We first show that $\T$ is contained in a horizontal cylinder. Indeed, if this is not the case then there would be a horizontal separatrix (that is a horizontal ray emanating from a point in $\Sig$) that intersects $\inter(s_1)$. It follows that there exists an embedded triangle bounded by $s_1$ and two other saddle connections $s'_2,s'_3$, where the slope $k'_2$ of $s'_2$ is negative, and the slope $k'_3$ of $s'_3$ is positive.
By definition, $\Pi(M,\Sig)$ contains the geodesic joining $k'_1$ and $k'_2$. But the Farey tessellation does not contain any such geodesic, hence we have a contradiction.

\begin{figure}[htb]
\begin{tikzpicture}[scale=0.45]
\draw (-4,4) -- (4,4) (-4,0) -- (4,0) (0,0) -- (0,8)-- (4,8) -- (4,0) (-4,4) -- (-4,0) -- (4,8) (0,0) -- (4,4);
\draw[dashed] (-6,4) -- (-4,4) (-6,0) -- (-4,0) (4,4) -- (7,4) (4,0) -- (7,0) (0,8) -- (0,10) (0,0) -- (0,-2) (4,8) -- (4,10) (4,0) -- (4,-2);

\foreach \x in {(-4,0),(-4,4), (0,0), (0,4), (0,8), (4,0), (4,4), (4,8)} \filldraw[fill=black] \x circle (3pt);
\draw (4.5,2) node {$\tiny s_1$} (1.5,2.1) node {$\tiny s_2$} (2,-0.5) node {$\tiny s_3$} (2,4.5) node {$\tiny s_4$} (-0.5,2) node {$\tiny s_5$} (-2,-0.5) node {$\tiny s_6$};

\draw (3,1)  node {$\tiny \T$} (1,3) node {$\tiny \T'$} (-1.5,1) node {$\tiny \T''$};
\draw (6,2) node {$\tiny C$} (2,9.5) node {$\tiny C'$};

\end{tikzpicture}
\caption{Tiling of a surface such that the associated tessellation is the Farey tessellation}
\label{fig:sq:tiled:Farey:tess}
\end{figure}

Let $C$ be the horizontal cylinder containing $\T$.
We can assume that $s_3$ is contained in the bottom border of $C$.
The top border of $C$ contains a horizontal saddle connection $s_4$ whose right endpoint coincides with the vertex of $\T$ opposite to $s_3$.
Observe that there is an embedded triangle $\T'$  bounded by $s_2,s_4$, and a third saddle connection denoted by $s_5$, which is contained in $C$.

Let $k_5$ be the slope of $s_5$. Since the slope of $s_2$ is $1$, and the slope of $s_4$ is $\infty$, the tessellation $\Pi(M,\Sig)$ contains a geodesic from $k_5$ to $1$ and a geodesic from $k_5$ to $\infty$. But $\Pi(M,\Sig)$ is the Farey tessellation, thus either $k_5=0$, or $k_5=2$. By construction, it is clear that $k_5 < 1$. Therefore, we have $k_5=0$, that is $s_5$ is vertical. This means that $S:=\T\cup\T'$ is an embedded unit square whose vertices are contained in $\Sig$.

If the square $S$ does not fill $C$,  then the bottom border of $C$ must contain saddle connection $s_6$ whose right endpoint coincide with the left endpoint of $s_3$. Consider the embedded triangle $\T''$ which is bounded by $s_6,s_5$, and a third saddle connection contained in $C$.
Repeating the arguments above, we see that the slope of the third side of $\T''$ is $1$.
By induction, we conclude that $C$ is filled by embedded unit squares with vertices in $\Sig$.

By the same argument, we also conclude  that $S$ is also contained in a vertical cylinder $C'$, which is filled out by unit squares with vertices in $\Sig$.
The same argument implies that every square in $C$ and $C'$ is contained in the intersection of a horizontal cylinder and a vertical cylinder, both of which are filled by unit squares with vertices in $\Sig$.
By connectedness, we deduce that $M$ is filled by such squares, which means that $(M,\Sig)$ is a translation cover of the torus $(\C/\Z^2,\{0\})$.
\end{proof}

\section{The graph of periodic directions}\label{sec:graph:per:dir}
Throughout this section, to simplify the discussion, $(M,\Sig)$ will be translation surface satisfying the topological Veech dichotomy. By Lemma~\ref{lm:blced:cov:same:tria}, the results in this section also hold in the case $(M,\Sig)$ is a half-translation surface.
We also normalize $(M,\Sig)$ using $\GL^+(2,\R)$ such that $\Aa(M)=1$.

In \textsection~\ref{sec:per:dir:graph}, we have defined a graph $\GG$ associated to $(M,\Sig)$. Recall that the vertices of $\GG$ are elements of $\Cc\sqcup \Ic$, and every edge of $\GG$ must join an element $\Delta$ of $\Ic$ to an element $k$ of $\Cc$ which is a vertex of $\Delta$. The length of every edge is set to be $\frac{1}{2}$. The distance on $\GG$ is denoted by $\dist$.
By construction, the graph  $\GG$ has the following properties:
\begin{itemize}
 \item[a)] Every vertex representing an element of $\Ic$ is the common endpoint of exactly $3$ (distinct) edges.


 \item[b)] Let $k,k'$ be two elements of $\Cc$. Then $\dist(k,k')=1$ if and only if there is an ideal triangle in $\Ic$ that contains $k$ and $k'$ as vertices.
 Equivalently, $\dist(k,k')=1$ if and only if there is a geodesic in $\Lc$ that joins $k$ and $k'$.
\end{itemize}

\subsection{Connectedness}\label{sec:G:connect}
For each $k \in \Cc$, let us denote by $\Sb(k)$ the set of saddle connections in the direction $k$.
The union of the saddle connections in $\Sb(k)$ will be denoted by $\hat{\Sb}(k)$.
Given $k,k'$ in $\Cc$, we define the {\em ordered intersection number} of the pair $(k,k')$ by
$$
\ii(k,k')=\min\{\#\left(\inter(s)\cap\hat{\Sb}(k')\right), \; s \in \Sb(k)\}.
$$
Note that the function $\ii$ is not symmetric, that is $\ii(k,k')$ and $\ii(k',k)$ might not be equal.
\begin{Proposition}\label{prop:G:dist:inters}
Let $k,k'$ be two directions in $\Cc$. Then
\begin{equation}\label{eq:G:dist:inters}
\dist(k,k')\leq \log_2(\min\{\ii(k,k'),\ii(k',k)\}+1)+1.
\end{equation}
In particular, the graph $\GG$ is connected.
\end{Proposition}

We first show
\begin{Lemma}\label{lm:G:inters:zero}
If $\min\{\ii(k,k'),\ii(k',k)\}=0$ then $\dist(k,k')=1$.
\end{Lemma}
\begin{proof}
Without loss of generality, we can assume $k=0, k'=\infty$, and  that $\ii(k,k')=0$. This means that $k$ is the vertical direction, $k'$ is the horizontal direction, and there is a vertical saddle connection $s$ which is not crossed by any horizontal saddle connection.
By assumption, the horizontal direction is periodic. Since $s$ does not intersect any horizontal saddle connection in its interior, $s$ must be contained entirely in a horizontal cylinder $C$.
There exists an embedded triangle contained in $C$ whose boundary contains $s$ and a horizontal saddle connection in the boundary of $C$.
It follows that $\Ic$ contains an ideal hyperbolic triangle with vertices $(0,\infty,k'')$, which means that, as vertices of $\GG$, $0$ and $\infty$ are connected by a path of length one.
\end{proof}

\subsection*{Proof of Proposition~\ref{prop:G:dist:inters}}
\begin{proof}
Again, without loss of generality, we can assume that $k=0,k'=\infty$, and $\ii(k,k') \leq \ii(k',k)$.
Let $n=\ii(k,k')=\min\{\ii(k,k'),\ii(k',k)\}$. If $n=0$, then by Lemma~\ref{lm:G:inters:zero}, we have $\dist(0,\infty)=1$. Let us suppose that $n>0$.

Consider a vertical saddle connection $s$ such that $\#\{\inter(s)\cap \hat{\Sb}(\infty)\}=n$.
Let us denote the horizontal saddle connections of $M$ by $a_1,\dots,a_m$. We choose the orientation of those saddle connections to be from the left to the right.
For each $a_i$, let $r_i$ be the distance along $a_i$ from its left endpoint to its first intersection with $\inter(s)$.
If $a_i\cap\inter(s)=\vide$, we set $r_i=+\infty$.

Assume that $r_1=\min \{r_1,\dots,r_m\}$. Let $a'_1$ be the subsegment of $a_1$ between its left endpoint and its first intersection with $\inter(s)$.
Using the developing map of the flat metric structure, we see that  $a'_1$ is contained in an embedded triangle $\T$ bordered by $s$ and two other saddle connections $s_1,s_2$.
Let $k_1,k_2$ be the directions of $s_1$ and $s_2$ respectively. By definition there is an ideal hyperbolic triangle with vertices $(0,k_1,k_2)$ in $\Ic$.
Thus  we have $\dist(0,k_1)=\dist(0,k_2)=1$ as vertices of $\GG$.
We now observe that
$$
\#\{\inter(s)\cap\hat{\Sb}(\infty)\}=\#\{\inter(s_1)\cap\hat{\Sb}(\infty)\}+\#\{\inter(s_2)\cap\hat{\Sb}(\infty)\}-1
$$
Hence $\min\{\ii(k_1,k'),\ii(k_2,k')\} < \ii(k,k')/2$. Replacing $k$ by either $k_1$ or $k_2$,  by induction, we get the desired conclusion.
\end{proof}

\subsection{Action of the Veech group}
Since an affine automorphism must send saddle connections to  saddle connections and embedded triangles to embedded triangles, we have an action of the group $\G$ on $\GG$ by automorphisms.

\begin{Lemma}\label{lm:Gam:act:freely}
The group $\G$ acts freely on the set of edges of $\GG$.
\end{Lemma}
\begin{proof}
Let $g$ be a an element of $\G$. Assume that $g$ fixes an edge of $e$ of $\GG$.
Recall that by construction, one endpoint of $e$ corresponds to an ideal triangle $\Delta$ in $\Ic$, and the other endpoint corresponds to a vertex $k$ of  $\Delta$.
Since $g$ fixes $e$, it must fix $\Delta$ and $k$ (this is because $g$ preserves each of the sets $\Ic$ and $\Cc$).
In particular, $g$ permutes the vertices of $\Delta$.
But since $g$ preserves the orientation of $\RP^1\simeq \partial \Hh$, if it fixes one vertex of $\Delta$, it must fix all of its vertices.
Therefore we must have $g=\pm \Id$.
\end{proof}

Recall that  $\ol{\Cc},\ol{\Lc},\ol{\Ic},\ol{\GG}$ are the quotients of
$\allowbreak{\Cc,\Lc,\Ic,\GG}$ by $\G$ respectively.

\begin{Proposition}\label{prop:Vee:quot:finite}
 If $(M,\Sig)$ is a Veech surface then the quotients $\ol{\Cc}$, $\ol{\Lc}$, and $\ol{\Ic}$ are all finite.
 In particular, $\ol{\GG}$ is a finite graph.
\end{Proposition}
\begin{proof}
Since every $\G$-orbit in $\Cc$ is a cusp of the corresponding Teichm\"uller curve, we draw that the quotient $\ol{\Cc}$ is finite.

Let us show that $\ol{\Lc}$ is finite.  Let $k$ be an element of $\Cc$.
We can assume that $k=\infty$, that is $k$ is the horizontal direction.
Since $M$ is a Veech surface, it is horizontally periodic.
Moreover, there is a matrix $A = \left(\begin{smallmatrix} 1 & c \\ 0 & 1 \end{smallmatrix} \right) \in \G$ such that the stabilizer of $\infty$ in $\G$ equals $\{A^n, \, n \in \Z\}$.
Without loss of generality, we can assume that $c>0$.

Let $\delta$ be the length of the shortest horizontal saddle connections of $(M,\Sig)$. Consider a geodesic $\g \in \Lc$ joining $\infty$ to a point $k'\in \R$. By definition, there is an embedded triangle $\T \in \Tria$ whose boundary contains a horizontal saddle connection $s$, and a saddle connection $s'$ in direction $k'$.

We first notice that $|s| \geq \delta$. Let $x'+\imath y'$, with $y'>0$, be the period of $s'$.  Since $\Aa(\T) \leq \Aa(M)=1$, we have $y' \leq 2/|s| \leq 2/\delta$. There exists $n \in \Z$ such that $0\leq  x'+ncy' \leq cy' \leq 2c/\delta$. Thus, up to the action of $\{A^n, \; n \in \Z\}$, we can assume that $0 \leq x' \leq 2c/\delta$. It follows that  $|s'|$ is bounded by $\frac{2}{\delta}\sqrt{1+c^2} $, which implies that $s'$ belongs to a finite set. Hence, up to the action of $\{A^n, \; n \in \Z\}$, there are only finitely many geodesics in $\Lc$ that contains $\infty$ as an endpoint.
Since  $\ol{\Cc}$ is finite, we conclude that the set $\ol{\Lc}$ is also finite.

We now claim that any geodesic $\g$ in $\Lc$ is contained in finitely many ideal triangles in $\Ic$.
Without loss of generality, we can assume that $\g$ is the upper half of the imaginary axis.
Let $\Delta$ be an ideal triangle in $\Ic$ that contains $\g$. By definition, $\Delta$ corresponds to an embedded triangle $\T \in \Tria$ whose boundary contains a horizontal saddle connection $s$, and a vertical saddle connection $s'$. Note that the direction of the third side of $\T$ is determined  up to sign by $|s|/|s'|$. Since there are only finitely many horizontal (resp. vertical) saddle connections, such a triangle belongs to a finite set. Therefore, there are only finitely many elements of $\Ic$ that contain $\g$.

Pick a representative for each $\G$-orbit in $\Lc$, and let $\Lc^*$ be the resulting  finite family of geodesics in  $\Hh$.
By the previous claim, the sets of triangles in $\Ic$ that contain at least one element of $\Lc^*$ is finite.
Since every ideal triangle in $\Ic$ is mapped by an element of $\G$ to a triangle that contains a geodesic in the family $\Lc^*$, we conclude that $\ol{\Ic}$ is finite.
\end{proof}


\subsection{Proof of Theorem~\ref{thm:area:TC:bound}}
\begin{proof}
Pick a representative element for each $\G$-orbit in $\Ic$.
Denote by $\Ic^*$ the resulting family.
Let $\F\subset \Hh$ be the union of the ideal triangles in $\Ic^*$.
Let $\varphi: \Hh \to \Hh/\G$ denote the canonical projection.
By Lemma~\ref{lm:ideal:tri:cover}, we have that $\Hh=\cup_{h\in \G} h(\F)$.
Therefore, $\varphi(\F)$ covers $\Hh/\G$.
For each ideal triangle $\Delta$ in $\Ic^*$, we have
$$
\Aa(\varphi(\Delta))\leq \Aa(\Delta) =\pi.
$$
Therefore
$$
\Aa(\Hh/\G) \leq \sum_{\Delta \in \Ic^*}\Aa(\varphi(\Delta)) \leq \#\ol{\Ic}\cdot\pi
$$
which proves \eqref{eq:lattice:vol:bound}.
If $\#\ol{\Ic}$ is finite then $\Aa(\Hh/\G) < \infty$ (by \eqref{eq:lattice:vol:bound}), which means that $\G$ is a lattice in $\PSL(2,\R)$, and hence $(M,\Sig)$ is a Veech surface.
Conversely, if $(M,\Sigma)$ is a Veech surface, then it  follows from Proposition~\ref{prop:Vee:quot:finite} that $\#\ol{\Ic}$ is finite.
\end{proof}


\section{Geometry of the graph of periodic directions}\label{sec:hyperbolicity}
Our goal now is to give the proof of Theorem~\ref{thm:G:per:dir:prop}.
Throughout this section $(M,\Sig)$ will be a half-translation surface satisfying the topological Veech dichotomy, which needs not to be a Veech surface.
\subsection{Infinite diameter}
In this section we will show
\begin{Proposition}\label{prop:inf:diam}
The graph $\GG$ has infinite diameter.
\end{Proposition}
To prove Proposition~\ref{prop:inf:diam}, we will make use of the connection between $\GG$ and the arc and curve graph on a surface with marked points.

\subsubsection{Arc and curve graphs} Let $S$ be a topological surface homeomorphic to $M \setm\Sig$. We will consider $S$ as a compact surface $\hat{S}$ with a finite set $V$ removed, points in $V$ are called punctures.   A simple closed curve in $S$ is {\em non-essential} if it is either homotopic to the constant loop, or bounds a disc that contains only one puncture. A simple arc in $S$ is a continuous map $\alpha : I \ra \hat{S}$,  where $I \subset \R$ is a compact interval, such that the restriction of $\alpha$ to $\inter(I)$ is an embedding and $\alpha(I)\cap V=\alpha(\partial I)$.
A simple arc is {\em non-essential} if it is homotopic relative to its endpoints to the constant map by a homotopy $H: I\times [0,1] \ra S$ such that for all $(t,s)\in \inter(I)\times[0,1)$, $H(t,s)\in S$. A simple closed curve or a simple arc is said to be {\em essential} if it is not non-essential.

Define the {\em curve graph} $\Curv(S)$  to be the  graph whose vertices are homotopy classes of essential simple closed curves in $S$, and there is an edge between two vertices if and only if the corresponding simple closed curves can be realized disjointly. Similarly, define the {\em arc and curve graph} $\ACurv(S)$ to be the graph whose vertices are homotopy classes of essential simple arcs and simple curves on $S$, and there is an edge between two vertices if and only if they can be realized disjointly in $S$.
We define the length of every edge of $\Curv(S)$ and of $\ACurv(S)$ to be one.
Denote by $\dd_{\rm AC}$ and $\dd_{\rm C}$ the distance in $\ACurv(S)$ and in $\Curv(S)$ respectively.
By construction, we have a natural embedding from $\Curv(C)$ into $\ACurv(S)$.

Since the graph of periodic directions $\GG$ is unchanged if we replace $(M,\Sig)$ by a half-translation covering, we can suppose that genus of $M$ (and hence the genus of $S$) is at least two.
We then have the following well known facts (see~\cite{MasSch13})
\begin{itemize}
 \item[$\bullet$] the  graphs $\Curv(S)$ and $\ACurv(S)$  are connected and have infinite diameter,

 \item[$\bullet$] the graphs $\Curv(S)$ and $\ACurv(S)$ are quasi-isometric.
\end{itemize}

A geodesic metric space  is said to be {\em Gromov hyperbolic} if there is a constant $\delta >0$ such that for any triple of points $(x,y,z)$ in this space, any geodesic from $x$ to $y$ is contained in the $\delta$-neighborhood of the union of a geodesic from $x$ to $z$ and a geodesic from $y$ to $z$.
By celebrated result of Masur-Minsky~\cite{MasMin99}, we know that $\Curv(S)$ (and hence $\ACurv(S)$) is Gromov hyperbolic.

Recall that a measured foliation on $S$ is by definition a measured foliation on $\hat{S}$ which has $k$-pronged singularities with $k\geq 3$ in $S$, and $k$-pronged singularities  with $k\geq 1$ at points in $V$ (see~\cite{Mosher:preprint}). In other word, measured foliations on $S$ are measured foliations on $\hat{S}$ that are modeled  by the foliations of  meromorphic quadratic differentials with at most simple poles.

A measured foliation is {\em minimal}  if all of its leaves are either dense in $\hat{S}$ or join two singularities, and there is no cycle of leaves.
In~\cite{Kla99}, Klarreich shows that the boundary at infinity $\partial_\infty\Curv(S)$ of $\Curv(S)$ can be identified with the space of topological minimal foliations on $S$.
Moreover, we have (see~\cite[Th. 1.4]{Kla99})

\begin{Theorem}[Klarreich]\label{th:Klar:conv:seq:infty}
 Given a minimal foliation $\mu$ on $S$, a sequence $(c_i)_{i\in \N} \subset \Curv(S)^{(0)}$ converges to the point in $\partial_\infty\Curv(S)$ represented by $\mu$
 if and only if for every accumulation point $\nu$ of $(c_i)_{i\in \N}$ in the space of projective measured foliations, $\nu$ is topologically equivalent to $\mu$
\end{Theorem}

Let $\iota$ denote the intersection number function on the space of measured foliations.
The following result is proved in~\cite{Rees81}.

\begin{Theorem}\label{thm:min:fol}
If $\lambda$ is a minimal measured foliation on $S$, then a measured lamination $\mu$ is topologically equivalent (Whitehead equivalent) to $\lambda$ if and only if $\iota(\lambda,\mu)=0$.
\end{Theorem}

For our purpose, we will also need the following result which is due to Smillie~\cite{Smi00} (see also~\cite{Vo05}).
\begin{Theorem}\label{thm:Smillie:Vorobets}
 Given any stratum of translation surfaces, there is a constant $K>0$ such that on any surface of area one in this stratum, there exists a cylinder of width bounded below by $K$.
\end{Theorem}
Since the area of a cylinder is equal to the product of its circumference and its width, if the surface has area one and the width of the cylinder is bounded below by $K$, then its circumference is at most $1/K$.
As a consequence of Theorem~\ref{th:Klar:conv:seq:infty}, we get the following (see also~\cite[Prop. 2.4]{Ham10})
\begin{Corollary}\label{cor:short:curv:min:fol}
For $t \in \R$, let $M_t$ denote the surface $a_t\cdot M$, where $a_t=\left(\begin{smallmatrix} e^t & 0 \\ 0 & e^{-t} \end{smallmatrix} \right)$. For $n \in \{0,1,\dots\}$, let $c_n$ be a regular geodesic on $M_n$ of length at most $1/K$. The existence of such a geodesic is guaranteed by Theorem~\ref{thm:Smillie:Vorobets}.
We consider $(c_n)$ as a sequence of vertices of $\Curv(S)$ via a homeomorphism $f: \hat{S} \ra M$ sending $V$ onto $\Sig$.

Assume that the vertical foliation $\mu$ on $M$ is minimal. Then the sequence $(c_n)$ defines a point in  $\partial_\infty\Curv(S)$. In particular,
$$
\lim_{n \ra \infty} \dd_{\rm C}(c_0, c_n)=\infty.
$$
\end{Corollary}
\begin{proof}
We can consider $\mu$ as an element of $\mathcal{MF}(S)$. Since $\mu$ is minimal by assumption, it represents a point in the boundary at infinity of $\Curv(S)$.
Let $\mu_t$ denote the (measured) foliation in the vertical direction on $M_t$. By definition, we have $\mu_t=e^t\cdot\mu$.
We have
$$
e^n\cdot \iota(c_n,\mu) = \iota(c_n,\mu_n) \leq |c_n| \leq 1/K, \text{ for all } n \in \N.
$$
Thus $\lim_{n\ra \infty}\iota(c_n,\mu) =0$. If $\nu$ is an element of $\mathcal{MF}(S)$ representing an accumulation point of $(c_n)$ in the space of projective measured foliations, then $\iota(\nu,\mu)=0$. Since $\mu$ is minimal, by Theorem~\ref{thm:min:fol}, $\nu$ is topologically equivalent to $\mu$. It follows from Theorem~\ref{th:Klar:conv:seq:infty} that $\mu$ is the limit of $(c_n)$ in $\partial_\infty\Curv(S)$, and the corollary follows.
\end{proof}

\subsubsection{Maps to the  curve complex and the arc and curve complex}
Let us fix a homeomorphism $f: \hat{S} \ra M$ such that $f^{-1}(\Sig)=V$.
Via the map $f$, we have two natural ``coarse'' mappings $\Psi: \Cc \ra \ACurv(S)$ and $\Psi': \Cc \ra \Curv(S)$ defined as follows: for any $k \in \Cc$,
\begin{itemize}
 \item[-] $\Psi(k)$ is the set  of vertices of $\ACurv(S)$ representing the homotopy classes of the saddle connections and regular geodesics (cylinders) in the direction $k$, and

 \item[-] $\Psi'(k)$ is the set vertices of $\Curv(S)$ representing the homotopy classes of the regular geodesics in the direction $k$.
\end{itemize}
By construction, $\diam\Psi(k)=\diam\Psi'(k)=1$ for any $k \in \Cc$.

\begin{Lemma}\label{lm:compare:leng:G:n:AC}
 Let $p,q$ be two periodic directions in $\Cc$ considered as vertices of $\GG$. Then
 \begin{equation}\label{eq:dist:G:n:AC}
  \dist(p,q) \geq \frac{1}{2} \dd_{\rm AC}(\Psi(p),\Psi(q)).
 \end{equation}
\end{Lemma}
\begin{proof}
 Let $\beta$ be a path of minimal length from $p$ to $q$ in $\GG$. Let $p=k_0,k_1,\dots,k_\ell=q$ be the elements of $\Cc$ that are contained in $\beta$, where $\dist(p,k_i)=i$.
 By construction, for each $i$, there are an element of $\Psi(k_i)$ and an element of $\Psi(k_{i+1})$ which are represented by two disjoint arcs in $S$. Therefore, there is an edge in $\ACurv(S)$ between a point in $\Psi(k_i)$ and a point in $\Psi(k_{i+1})$.
 Since $\diam\Psi(k_i)=1$, it follows that there is a path from a point in $\Psi(p)$ to a point in $\Psi(q)$ of length at most $2\ell$, from which we get  inequality~\eqref{eq:dist:G:n:AC}.
\end{proof}

\subsubsection{Proof of Proposition~\ref{prop:inf:diam}}
\begin{proof}
By Lemma~\ref{lm:compare:leng:G:n:AC}, it is  enough to show that  $\diam\Psi(\Cc)=\infty$, which is equivalent to $\diam\Psi'(\Cc)=\infty$ because the embedding of $\Curv(S)$ into $\ACurv(S)$ is a quasi-isometry. Since we can rotate $M$ such that the vertical foliation is minimal, this follows immediately from Corollary~\ref{cor:short:curv:min:fol}.
\end{proof}

\subsection{Hyperbolicity}
Our goal now is to show
\begin{Proposition}\label{prop:G:hyperbolic}
The graph $\GG$ is Gromov hyperbolic.
\end{Proposition}

For this purpose,  we will use the following criterion by Masur-Schleimer~\cite{MasSch13}.

\begin{Theorem}[Masur-Schleimer]\label{thm:hyp:crit:Ma-Sc}
 Suppose that $\mathcal{X}$ is a graph with all edge lengths equal to one. Then $\mathcal{X}$ is Gromov hyperbolic if there is a constant $R \geq 0$, and for all unordered pair of vertices $x,y$ in $\mathcal{X}^0$, there is a connected subgraph $g_{x,y}$ containing $x$ and $y$ with the following properties

\begin{itemize}
  \item[$\bullet$] (Local) If  $d_\mathcal{X}(x,y) \leq 1$ then $g_{x,y}$ has diameter at most $R$,

  \item[$\bullet$] (Slim triangle) For any $x,y,z \in \mathcal{X}^0$, the subgraph $g_{x,y}$ is contained in the $R$-neighborhood of $g_{x,z}\cup g_{z,y}$.
\end{itemize}
\end{Theorem}

We will also need the following improvement of Proposition~\ref{prop:G:dist:inters}.
\begin{Lemma}\label{lm:dist:G:inters:bound}
There exists a constant $\kappa_0$  depending on the stratum of $(M,\Sig)$ such that, for any pair of  saddle connections $s_1$ and $s_2$ of $(M,\Sig)$ with directions $k_1$ and $k_2$ respectively, we have
\begin{equation}\label{eq:dist:G:inters:bound}
 \dist(k_1,k_2) \leq \log_2(\#(\inter(s_1)\cap\inter(s_2))+1) +\kappa_0.
\end{equation}
\end{Lemma}
\begin{proof}
 Assume first that $\#(\inter(s_1)\cap\inter(s_2))=0$, which means that $s_1$ and $s_2$ are disjoint. We can then add other saddle connections to the family $\{s_1,s_2\}$ to obtain a triangulation of $(M,\Sig)$. Let $\kappa_0$ be the number of triangles in this triangulation.  Note that this number only depends on the stratum of $(M,\Sig)$.  Now, since each triangle in this triangulation represents a vertex in $\GG$ that is connected to the vertices representing the directions of its three sides, we see that there is a path in $\GG$ from $k_1$ to $k_2$ of length at most $\kappa_0$. Thus we have
 $$
 \dist(k_1,k_2) \leq \kappa_0.
 $$
 For the case $\#(\inter(s_1)\cap\inter(s_2))>0$,  we us the same induction as in Proposition~\ref{prop:G:dist:inters} to conclude.
\end{proof}
\begin{Corollary}\label{cor:cyls:width:bounded}
Let $C$  be a cylinder, and $s$ a saddle connection in $(M,\Sig)$. Let $w(C)$ denote the width of $C$ and $|s|$ the length of $s$.  Then the distance in $\GG$ between the direction of $C$ and the direction of $s$ is at most $\log_2(\frac{|s|}{w(C)}+1)+\kappa_0$.
\end{Corollary}
\begin{proof}
Let $c$ be a core curve of $C$. Let $m$ be the number of intersections between $c$ and $\inter(s)$.  Obviously, we only need to consider the case $c$ and $s$ are not parallel.
Since $|s| \geq  mw(C)$, we have $m \leq |s|/w(C)$. If $s'$ is a saddle connection in the boundary of $C$, then we have
$$
\#(\inter(s),\inter(s')) \leq m \leq \frac{|s|}{w(C)}.
$$
We then conclude by Lemma~\ref{lm:dist:G:inters:bound}.
\end{proof}

\subsubsection{Paths connecting pairs of points in $\Cc$}
In view of Theorem~\ref{thm:hyp:crit:Ma-Sc}, to simplify the arguments,  we will consider another graph, denoted by $\GG'$, closely related to $\GG$.  The vertices of $\GG'$ are elements of $\Cc$. Two vertices are connected by an edge if and only if they are two vertices of an ideal triangle in $\Ic$. The length of every edge is set to be one.

There is a natural map $\Xi: \GG' \to \GG$ defined as follows: $\Xi$ is identity on $\Cc \simeq {\GG'}^{(0)}$. For each edge $e\in \GG'^{(1)}$, whose endpoints are $k_1,k_2\in \Cc$, $\Xi(e)$ is the union of two edges in $\GG$ that connect $k_1,k_2$ through a vertex representing an ideal triangle  $\Delta\in \Ic$.
Recall that by construction,  $k_1,k_2$ are two vertices of $\Delta$.

Note that $\Delta$ may not be unique, however the number of admissible $\Delta$ is bounded by a constant depending only on the stratum of $(M,\Sig)$.
Indeed, let $\T$ be an embedded triangle  associated with $\Delta$ whose sides are denoted by $s_1,s_2,s_3$.
We can assume that the directions of $s_1$ and $s_2$ are $k_1$ and $k_2$ respectively.
Since each pair of oriented saddle connections are contained in the boundary of at most one embedded triangle, and the number of saddle connections in a given direction is determined by the stratum of $(M,\Sig)$,
the number of ideal triangle in $\Ic$ that contains $k_1$ and $k_2$ as vertices is bounded by a universal constant depending only on the stratum of $(M,\Sig)$.

Since every element of $\Ic$ is of distance $\frac{1}{2}$ from $\Cc$, we get
\begin{Lemma}~\label{lm:GG:GGb:equiv}
For any pair $(k,k')$ of directions in $\Cc$, the distances between $k$ and $k'$ in $\GG'$ and in $\GG$ are the same. The map $\Xi$ is therefore a quasi-isometry.
\end{Lemma}

Our goal now is to show that $\GG'$ is Gromov hyperbolic.  Lemma~\ref{lm:GG:GGb:equiv} then implies that $\GG$ is also Gromov hyperbolic.
By a slight abuse of notation, we will also denote by $\dist$ the distance in $\GG'$.

Our first task is to construct for every pair $(k,k')$ of directions in $\Cc$ a path in $\GG'$ connecting them.
Using $\PSL(2,\R)$, we can assume that $k$ is the horizontal direction  and $k'$ is the vertical direction. We can further normalize $M$ by a matrix $a_t:=\left(\begin{smallmatrix} e^t & 0 \\ 0 & e^{-t}\end{smallmatrix}\right), \; t \in \R$, such that the shortest horizontal saddle connection and the shortest vertical saddle connection have the same length.

For any $t \in \R$, let $M_t:=a_t \cdot M$. If $c$ is a regular geodesic or a saddle connection on $(M,\Sig)$, the length of $c$ on $M_t$ will be denoted by $|c|_t$.
By Theorem~\ref{thm:Smillie:Vorobets}, there is a cylinder $C_t$ on $M_t$ of width bounded below by $K$.  The cylinder $C_t$ may be not unique, but  we have
\begin{Lemma}\label{lm:2:cyl:bdd:width}
If $C'_t$ is another cylinder of width bounded below by $K$ in $M_t$, then the distance in $\GG'$ between the directions of $C_t$ and $C'_t$  is at most $(\log_2(K^{-2}+1)+\kappa_0)$.
\end{Lemma}
\begin{proof}
Since $\Aa(M_t)=\Aa(M)=1$,  the circumference of $C_t$ is at most $K^{-1}$. In particular, a saddle connection $s$ in the boundary of $C_t$ has length at most $K^{-1}$.
Let $k$ and $k'$ be the directions of $C_t$ and $C'_t$ respectively. Then Corollary~\ref{cor:cyls:width:bounded} implies
$$
\dist(k,k') \leq \log_2(\frac{|s|_t}{w(C'_t)}+1)+\kappa_0 \leq \log_2(K^{-2}+1)+\kappa_0.
$$
\end{proof}
In what follows, for any $t\in \R$, we denote by $C^0_t$ a cylinder of width bounded below by $K$ in $M_t$ and by $k(t)$ the direction of $a_{-t}(C^0_t)$. Note that we have $k(t) \in \Cc$.

\begin{Lemma}\label{lm:seq:cyls:bounded:w}\hfill
\begin{itemize}
 \item[(i)] There exists $t_0>0$ such that if $t>t_0$, then $k(t)=0$, and if $t<-t_0$ then $k(t)=\infty$.

 \item[(ii)] For any $t_1,t_2\in \R$, $\dist(k({t_1}),k({t_2})) \leq \log_2(K^{-2}+1)+\frac{|t_1-t_2|}{\ln(2)}+\kappa_0$.
\end{itemize}
\end{Lemma}
\begin{proof}\hfill
\begin{itemize}
 \item[(i)] If $t>0$ is large enough then the width of any vertical cylinder in $M_t$ is at least $1/K$. Thus a non-vertical cylinder in $M_t$ has circumference at least $1/K$, hence its width must be smaller than $K$. Thus we must have $k(t)=0$. Similar arguments apply  for $M_{-t}$.

 \item[(ii)] Observe that we have for any  saddle connection or  regular geodesic $c$ on $M$,
 $$
 \frac{|c|_{t_1}}{|c|_{t_2}} \leq e^{|t_1-t_2|}.
 $$
 Since the length of a core curve of $C^0_{t_1}$ on $M_{t_1}$ is at most $K^{-1}$, its length in $M_{t_2}$ is at most $e^{|t_1-t_2|}K^{-1}$. Thus the conclusion follows from Corollary~\ref{cor:cyls:width:bounded}.
\end{itemize}
\end{proof}
Define
\begin{equation}\label{eq:def:path:in:G:1}
\gg^*(k,k'):=\{k(i), \; i \in \Z\} \subset \Cc \simeq {\GG'}^{(0)}.
\end{equation}
By Lemma~\ref{lm:seq:cyls:bounded:w}, the set $\gg^*(k,k')$ is  finite.
For any $i \in \Z$, let $\gamma_i$ be a path of minimal length in $\GG'$ from $k(i)$ and $k({i+1})$.
Let
\begin{equation}\label{eq:def:path:in:G:2}
\gg(k,k'):=\bigcup_{i\in\Z} \gamma_i \subset \GG'.
\end{equation}
By construction, $\gg(k,k')$ is obviously a connected finite subgraph of $\GG'$.
For any subset $\mathcal{A}$ of $\GG'$ and any $r>0$, let us denote by $\Nc(\mathcal{A},r)$ the $r$-neighborhood of $\mathcal{A}$ in $\GG'$.

\begin{Lemma}\label{lm:gg:prop}
There is a constant $R_1>0$, depending only on the stratum of $(M,\Sig)$, such that
\begin{itemize}
 \item[(a)] $\gg(k,k') \subset \Nc(\gg^*(k,k'),R_1)$, and

 \item[(b)] for any $t \in \R$, $k(t) \in \Nc(\gg^*(k,k'),R_1)$.
\end{itemize}
\end{Lemma}
\begin{proof}
 Set $R_1=\log_2(K^{-2}+1)+\kappa_0+1/\ln(2)$. From Lemma~\ref{lm:seq:cyls:bounded:w}, we have $\dist(k(i),k({i+1})) \leq R_1$. Thus every point in $\gamma_i$ is of distance at most $R_1/2$ from either $k(i)$ or $k({i+1})$, from which we get (a). Again, by  Lemma~\ref{lm:seq:cyls:bounded:w}, any $k(t)$ is of distance at most $R_1$ from a point $k(i)$, with $i\in \Z$, and (b) follows.
\end{proof}

\subsubsection{Local property}
\begin{Lemma}\label{lm:loc:prop:G}
There is a constant $R_2>0$ such that if $\dist(k,k')=1$, then $\diam(\gg(k,k'))< R_2$.
\end{Lemma}
\begin{proof}
We can suppose that $k$ is the horizontal direction, and $k'$ is the vertical direction.
By assumption, there are a horizontal saddle connection $s$ and a vertical saddle connection $s'$ that are two sides of an embedded triangle $\T$ in $M=M_0$.
Recall that $M$ is normalized so that the shortest horizontal saddle connection $s_0$, and the shortest vertical saddle connection $s'_0$ have the same length, say $\delta$.
We first have
$$
\delta^2 \leq |s||s'| =2 \Aa(\T) < 2.
$$
Thus $\delta < \sqrt{2}$.

For any $t\in \R$, the lengths of $s_0$  and $s'_0$ in $M_t$ are respectively $e^{t}\delta$ and $e^{-t}\delta$.
If $i<0$, then the length of $s_0$ in $M_i$ is smaller than $\sqrt{2}$.
It follows from Corollary~\ref{cor:cyls:width:bounded} that $\dist(k,k(i)) \leq \log_2(\sqrt{2}K^{-1}+1)+\kappa_0$.
Similarly, if $i>0$ then the length of $s'_0$ in $M_i$ is smaller than $\sqrt{2}$, thus $\dist(k',k(i)) \leq \log_2(\sqrt{2}K^{-1}+1)+\kappa_0$.
Therefore we have
$$
\gg^*(k,k') \subset \Nc(\{k,k'\},\log_2(\sqrt{2}K^{-1}+1)+\kappa_0).
$$
From Lemma~\ref{lm:gg:prop} (a), we get
$$
\diam(\gg(k,k')) \leq R_2
$$
with $R_2=2(R_1+\log_2(\sqrt{2}K^{-1}+1)+\kappa_0)+1$.
\end{proof}

\subsubsection{Slim triangle property}
Let $(k,k')$ be  a pair of directions in $\Cc$.
We use $\PSL(2,\R)$ to transform $k$ to the horizontal direction, $k'$ to the vertical direction, and such that the shortest horizontal and vertical saddle connections have the same length.

For any $t\in \R$, and $R\in (0,+\infty)$, let $\LL_t(k,k',R)\subset \Cc$ denote the set of directions of the cylinders whose circumference in $M_t$ is at most $R$.   Since $M_t$ always contains a cylinder of width bounded below by $K$ (hence its circumference is at most $K^{-1}$), for any $R> K^{-1}$, the set $\LL_t(k,k',R)$ is non-empty.
Define
\begin{equation}\label{eq:def:path:in:G:3}
\hat{\gg}^*(k,k')=\bigcup_{t\in\R} \LL_t(k,k',2K^{-1}) \subset {\GG'}^{(0)}.
\end{equation}
By construction, we have $\gg^*(k,k')\subset \hat{\gg}^*(k,k')$.

\begin{Lemma}\label{lm:enlarged:connect:path}
The set $\hat{\gg}^*(k,k')$ is finite, and there exists a constant $R_3$ such that
$$
\hat{\gg}^*(k,k') \subset \Nc(\gg^*(k,k'),R_3).
$$
\end{Lemma}
\begin{proof}
Observe that for any  regular geodesic $c$ in $M$,  and any $t_1,t_2 \in \R$, we have $|c|_{t_1}/|c|_{t_2} \leq e^{|t_1-t_2|}$. It follows that
$$
\hat{\gg}^*(k,k') \subset \bigcup_{i\in \Z}\LL_i(k,k',2eK^{-1}).
$$
For $i>0$ large enough, we have $\LL_i(k,k',2eK^{-1})=\{k'\}$, and $\LL_{-i}(k,k',2eK^{-1})=\{\k\}$. Since for any fixed $i$, the set of cylinders with circumference at most $2eK^{-1}$ on $M_i$ is finite, we draw that $\hat{\gg}^*(k,k')$ is a finite set.

Now, by Corollary~\ref{cor:cyls:width:bounded},  the direction of a cylinder with circumference at most $2eK^{-1}$ on $M_i$ is of distance at most $\log_2(2eK^{-2}+1)+\kappa_0$ from $k(i)$. Therefore
$$
\hat{\gg}^*(k,k') \subset \Nc(\gg^*(k,k'),R_3)
$$
with $R_3=\log_2(2eK^{-2}+1)+\kappa_0$.
\end{proof}

We now show
\begin{Lemma}~\label{lm:slim:tria:coarse}
There is a constant $R_4>0$ such that for any triple $(k,k',k'')$ of directions in $\Cc(M,\Sig)$, we have
$$
\hat{\gg}^{*}(k,k') \subset \Nc(\hat{\gg}^*(k,k'')\cup\hat{\gg}^*(k',k''),R_4).
$$
\end{Lemma}
\begin{proof}
We can  renormalize $M$ (using $\PSL(2,\R)$) such that $(k,k',k'')=(\infty,0,1)$.
Note that this normalization is not necessarily the normalization used to define the path $\gg(k,k')$ in \eqref{eq:def:path:in:G:2}.
In particular, $\gg^*(k,k')$ does not necessarily equal the set $\{k(i), \; i\in \Z\}$.
Nevertheless, we obtain the same subset $\hat{\gg}^*(k,k')$ by \eqref{eq:def:path:in:G:3}, that is $\hat{\gg}^*(k,k')=\cup_{t\in\R}\LL_t(k,k',2K^{-1})$.

Consider a direction $\hat{k}$ in $\hat{\gg}^*(k,k')$. By definition, $\hat{k}$ is the direction of a cylinder $C$ whose circumference in $M_t:=a_t\cdot M$ is at most $2K^{-1}$ for some $t \in \R$.

\medskip

\noindent {\em Claim:} if $t \leq 0$ then $\hat{k}$ is contained in the $(\log_2(4K^{-2}+1)+\kappa_0)$-neighborhood of $\hat{\gg}^*(k,k'')$.
\begin{proof}[Proof of the claim]
Let $M':= U\cdot M$, where $U=\left(\begin{smallmatrix} 1 & -1 \\ 0 & 1 \end{smallmatrix} \right)$. Note that $U(k)=k=\infty$ and $U(k'')=0$.
By definition,  $\hat{\gg}^*(k,k'')$ is the set of directions $\hat{k}' \in \Cc$ such that, for some $s \in \R$, the circumference of a cylinder in direction $\hat{k}'$ is at most $2K^{-1}$ in  $a_s\cdot M'$.

We claim that, for $s=t$, the circumference of $C$ in $a_t\cdot M'$ is at most $4K^{-1}$. To see this, we observe that
$$
M'_t:=a_t\cdot M'= \left(a_t\cdot U \cdot a_{-t}\right) \cdot M_t.
$$
Recall that the circumference of $C$ in $M_t$ is at most $2K^{-1}$. Since $a_t\cdot U \cdot a_{-t}=\left(\begin{smallmatrix} 1 & -e^{2t} \\ 0 & 1 \end{smallmatrix} \right)$, and $t \leq 0$, it follows that the circumference of $C$ in $M'_t$ is at most $4K^{-1}$.

Let $D^0_t$ be a cylinder of width bounded below by $K$ in $M'_t$ (whose existence is guaranteed by Theorem~\ref{thm:Smillie:Vorobets}). By definition, the direction of $D^0_t$ belongs to $\hat{\gg}^*(k,k'')$. From Corollary~\ref{cor:cyls:width:bounded}, it follows that the distance between the directions of $C$ and $D^0_t$ is at most $\log_2(4K^{-2}+1)+\kappa_0$. The claim is then proved.
\end{proof}

It follows immediately from the claim that $\LL_t(k,k',2K^{-1})$ is contained in the $R_4$-neighborhood of $\hat{\gg}^*(k,k'')$ if $t \leq 0$, with
$$
R_4=\log_2(4K^{-2}+1)+\kappa_0.
$$
By similar arguments, one can also show that $\LL_t(k,k',2K^{-1})$ is contained in the $R_4$-neighborhood of $\hat{\gg}^*(k',k'')$ if $t\geq 0$.
The lemma is then proved.
\end{proof}

\begin{Corollary}\label{cor:slim:tria}
Let $R_5=R_1+R_3+R_4$, where $R_1,R_3,R_4$ are the constants of Lemmas~\ref{lm:gg:prop},\ref{lm:enlarged:connect:path},\ref{lm:slim:tria:coarse} respectively.
Then for any triple $(k,k',k'')$ of directions in $\Cc$, we have
$$
\gg(k,k') \subset \Nc(\gg(k,k'')\cup\gg(k',k''),R_5).
$$
\end{Corollary}
\begin{proof}
 It follows from Lemma~\ref{lm:slim:tria:coarse} that we have
 $$
 \hat{\gg}^{*}(k,k')\subset \Nc(\hat{\gg}^*(k,k'')\cup\hat{\gg}^*(k',k''),R_4).
 $$
 Since $\gg^*(k,k') \subset \hat{\gg}^*(k,k')$, Lemma~\ref{lm:enlarged:connect:path} implies
 $$
 \gg^*(k,k') \subset \Nc(\gg^*(k,k'')\cup\gg^*(k',k''),R_3+R_4) \subset \Nc(\gg(k,k'')\cup\gg(k',k''),R_3+R_4).
 $$
 Finally, from Lemma~\ref{lm:gg:prop}, we get
 $$
 \gg(k,k') \subset \Nc(\gg(k,k'')\cup\gg(k',k''),R_1+R_3+R_4).
 $$
\end{proof}

\subsubsection{Proof of Proposition~\ref{prop:G:hyperbolic}}
\begin{proof}
By Theorem~\ref{thm:hyp:crit:Ma-Sc}, Lemma~\ref{lm:loc:prop:G} and  Corollary~\ref{cor:slim:tria} imply that $\GG'$ is Gromov hyperbolic. Since $\GG'$ and $\GG$ are quasi-isometric (c.f. Lemma~\ref{lm:GG:GGb:equiv}), this shows that $\GG$ is Gromov hyperbolic.
\end{proof}

\subsection{Proof of Theorem~\ref{thm:G:per:dir:prop}}
\begin{proof}
 The first part of Theorem~\ref{thm:G:per:dir:prop} follows from Propositions~\ref{prop:G:dist:inters},\ref{prop:inf:diam},\ref{prop:G:hyperbolic}.
 By Lemma~\ref{lm:Gam:act:freely}, we know that $\G$ acts freely on the set of edges of $\GG$.

 Assume now that $\G$ is a lattice in $\PSL(2,\R)$, then $\ol{\GG}$ is a finite graph by Proposition~\ref{prop:Vee:quot:finite}.
 Conversely, if $\ol{\GG}$ is a finite graph then in particular $\ol{\Ic}$ is a finite set.
 Thus, $\G$ has finite covolume by \eqref{eq:lattice:vol:bound}, which means that $(M,\Sig)$ is a Veech surface.
\end{proof}


\appendix
\section{Quotient of the periodic direction graph and generating sets of the Veech group}\label{sec:fund:dom:n:gen:sets}
Throughout this section, we will suppose that $(M,\Sig)$ is a half-translation Veech surface and that $\Aa(M)=1$.
Our goal is to provide an algorithm to determine the graph $\ol{\GG}=\GG/\G$.
In particular, we get a representative families of $\Ic/\G$, hence an upper bound on the volume of the Teichm\"uller curve $\Hh/\G$ by \eqref{eq:area:bound}.
As a by product, we obtain an algorithm to find a generating set of the Veech group $\G$.
To lighten the notation, we will omit $(M,\Sig)$ from the notation of the objects constructed from the pair $(M,\Sig)$.

\subsection{Reference domain for a periodic direction.}\label{sec:dom:of:cusp}
Assume that $(M,\Sig)$ is a horizontally periodic,  we then say that $(M,\Sig)$ is {\em normalized} if the shortest horizontal saddle connection of $M$ has length equal to $1$.

Let $k$ be a periodic direction in $\Cc$. There is an element  $A \in \PSL(2,\R)$, determined up to the left action of $\{U_t=\left(\begin{smallmatrix} 1 & t \\ 0 & 1 \end{smallmatrix}\right), t \in \R \}$, such that $A(k)=\infty$, and $(M',\Sig'):=A\cdot(M,\Sig)$ is normalized.
Let $\G'$ denote the Veech group of $(M',\Sig')$.
Note that we have $\G'=A\cdot \G \cdot A^{-1}$.

There exists $a\in \R_{>0}$ such that the stabilizer $\Stab_{\G'}(\infty)$ of $\infty$ in $\G'$ equals $\{\left(\begin{smallmatrix} 1 & \Z a \\ 0 & 1 \end{smallmatrix} \right)\}$.
We will call $a$ the {\em period} of the direction $k$.
Note that  $a$ stays unchanged if we replace $A$ by $U_t\cdot A$.

Let $\Ic^*(M',\Sig',\infty)$ denote the set of hyperbolic ideal triangles $\Delta \in \Ic(M',\Sig')$ such that

\begin{itemize}
\item $\infty$ is a vertex of $\Delta$, and


\item $\Delta$ intersects the vertical strip $(0,a)\times\R_+$.
\end{itemize}
\begin{Lemma}\label{lm:tria:in:ref:domain}
Let $\kappa$ denote the length of the longest horizontal saddle connection in $(M',\Sig')$. Let $\Delta$ be an ideal triangle in $\Ic^*(M',\Sig',\infty)$,
and $\T \in \Tria(M',\Sig')$ an embedded triangle which gives rise to $\Delta$. Denote the sides of $\T$ by $s_0,s_1,s_2$, where $s_0$ is a
horizontal saddle connection. Then
$$
\min\{|s_1|,|s_2|\} \leq \max\{2\sqrt{1+a^2},\sqrt{4+\kappa^2/4} \}.
$$
\end{Lemma}
\begin{proof}
For $i=1,2$, let $x_i+\imath y_i \in \C$ be the period of $s_i$, and $k_i=\frac{x_i}{y_i}$.  We can always assume that $y_i>0$, and $k_1< k_2$. Note that we have $y_1=y_2$.
Since $\T$ is an embedded triangle $\Aa(\T)=\frac{1}{2}y_1|s_0| < 1$.  As $(M',\Sig')$ is normalized, $|s_0|\geq 1$, hence $y_1=y_2 < 2$.


By definition, $[k_1,k_2]$ intersects the interval $(0,a)$. We have two cases:
\begin{itemize}
  \item  If $(0,a) \not\subset [k_1,k_2]$, then at least one of the following holds:  $k_1\in (0,a)$ or $k_2\in (0,a)$. Assume that  $k_1 \in (0,a)$, then $0<x_1<ay_1 <2a$. It follows that $|s_1|<2\sqrt{1+a^2}$. By the same argument, if $k_2\in (0,a)$ then $|s_2| < 2\sqrt{1+a^2}$.

  \item If $(0,a) \subset [k_1,k_2]$ then  $k_1\leq 0 < a \leq k_2$. Note that in this case $|s_0|=x_2-x_1$. Since $x_1\leq 0 \leq x_2$, it follows $\min\{-x_1,x_2\} \leq \frac{|s_0|}{2} \leq \frac{\kappa}{2}$. Since $0<y_1=y_2<2$, we have $\min\{|s_1|,|s_2|\}\leq \sqrt{4+\kappa^2/4}$.
\end{itemize}
\end{proof}

\begin{Corollary}\label{cor:ref:tria:finite}
The set $\Ic^*(M',\Sig',\infty)$ is finite.
\end{Corollary}
\begin{proof}
Remark that an embedded triangle is uniquely determined by two of its oriented sides (the sides of a triangle are naturally endowed with the induced orientation). Therefore the number of triangles in $\Ic^*(M',\Sig',\infty)$ is bounded by the number of pairs $(s,s')$ of oriented saddle connections, where $s$ is horizontal, and $s'$ is non-horizontal with length at most $\max\{2\sqrt{1+a^2}, \sqrt{4+\kappa^2/4}\}$.
Since the set of saddle connections of length bounded by a constant is finite, it follows that the set $\Ic^*(M',\Sig',\infty)$ is finite.
\end{proof}

\begin{Remark}
  Lemma~\ref{lm:tria:in:ref:domain} provides us with a criterion for the search of ideal triangles in $\Ic^*(M',\Sig',\infty)$, namely, we only need to look for embedded triangles bounded by a horizontal saddle connections, and a non-horizontal saddle connection of length at most $\max\{2\sqrt{1+a^2},\sqrt{4+\kappa^2/4}\}$.
\end{Remark}

Let $\DD^*(M',\Sig',\infty)$ denote the union of the ideal triangles in $\Ic^*(M',\Sig',\infty)$.

\begin{Lemma}\label{lm:ref:domain} \hfill
\begin{itemize}
  \item[(a)] The domain $\DD^*(M',\Sig',\infty)$ is  connected.

 \item[(b)] For any hyperbolic ideal triangle $\Delta$ in $\Ic(M',\Sig')$ that has $\infty$ as a vertex, the $\Stab_{\G'}(\infty)$-orbit of $\Delta$ intersects the set $\Ic^*(M',\Sig',\infty)$.
\end{itemize}

\end{Lemma}
\begin{proof}\hfill
\begin{itemize}
\item[(a)] To show that $\DD^*(M',\Sig',\infty)$ is connected,  it  suffices to show that its projection $J$ to the real axis is connected, which means that $J$ is an interval.
By definition, the projection of any triangle in $\Ic^*(M',\Sig',\infty)$ to the real axis is an interval that intersects $(0,a)$.
Therefore, it is enough to show that $(0,a) \subset J$.

Let $k_0$ be any direction in $(0,a)$.
Consider the surface $U_{-k_0}\cdot (M',\Sig')$, where $U_{-k_0}=\left(\begin{smallmatrix} 1 & -k_0 \\ 0 & 1 \end{smallmatrix} \right)$.
Note that the action of $U_{-k_0}$ on $\R \subset \RP^1$ is the translation by $-k_0$.
Let $s$ be (one of) the longest horizontal saddle connection of $U_{-k_0}\cdot(M',\Sig')$.
This saddle connection is contained in the bottom border of a horizontal cylinder, say $C$.
There is a singularity in the top border of $C$ such that the downward vertical ray emanating from this singularity hits $s$ before exiting $C$.
Thus there is an embedded triangle in $C$ that contains $s$ as a side and the vertical segment above.
Let $s_1,s_2$ be the other sides of this triangle, and $k_1,k_2$ be the directions of $s_1$ and $s_2$ respectively.
We can assume that $k_1\leq 0 \leq k_2$.
Applying $U_{k_0}=\left(\begin{smallmatrix} 1 & k_0 \\ 0 & 1 \end{smallmatrix} \right)$, we get an embedded triangle in $(M',\Sig')$  which corresponds to the hyperbolic ideal triangle $\Delta$ with vertices $\infty,k_1+k_0,k_2+k_0$. Since $k_0 \in [k_1+k_0,k_2+k_0]\cap (0,a)$, we have $\Delta \in \Ic^*(M',\Sig',\infty)$. Thus $k_0\in [k_1,k_2] \subset J$, and we have $(0,a)\subset J$ as desired.

\medskip

\item[(b)] Let $\T \in \Tria(M',\Sig')$ be the embedded triangle corresponding to an ideal triangle $\Delta$ which has $\infty$ as a vertex.
Let $s_0,s_1,s_2$ denote the sides of $\T$, and $k_i\in \R\cup\{\infty\}$ the slope of $s_i$.
We can suppose that $k_0=\infty$ (that is $s_0$ is a horizontal saddle connection), and that $k_1 < k_2$.
Since the action of $\left(\begin{smallmatrix} 1 & a \\ 0 & 1 \end{smallmatrix} \right)$ on $\R$ is given by $x \mapsto x+a$, there exists $U\in \Stab_{\G'}(\infty)$ such that  $U(k_1) \in [0,a)$, which implies that $U(\Delta) \in \Ic^*(M',\Sig',\infty)$.
\end{itemize}
\end{proof}

We will call $\DD(k):=A^{-1}(\DD^*(M',\Sig',\infty))$ a {\em reference domain} for the direction $k$.
It follows from Lemma~\ref{lm:ref:domain} that $\DD(k)$ is a polygon in $\Hh$ with geodesic boundary, which is not necessarily convex.
Set $\Ic^*(k)=A^{-1}(\Ic^*(M',\Sig',\infty))$.
By definition,  $\Ic^*(k)$ is the set of ideal triangles in $\Tria(M,\Sig)$ that compose $\DD(k)$.
Let $\NN(k)$ denote the set
$$
\NN(k)=\{k' \in \Cc, \; k'\neq k, k' \text{ is a vertex of some triangle in } \Ic^*(k)\}.
$$
In other words, $\NN(k)$ is the set of vertices  of $\DD(k)$ in $\partial\Hh\setminus\{k\}$.
The following lemma is a reformulation of  Lemma~\ref{lm:ref:domain}.

\begin{Lemma}\label{lm:neigh:1}
Let $\Stab_\G(k)$ denote the stabilizer of $k$ in $\G$. We regard elements of $\Cc$ and $\Ic$ as vertices of $\GG$.
\begin{itemize}
 \item[(i)] For any $k'\in\NN(k)$, $\dist(k,k')=1$.

 \item[(ii)] If $\Delta\in \Ic$ such that $\dist(k,\Delta)=\frac{1}{2}$, then the domain $\DD(k)$ contains an ideal triangle in the $\Stab_\G(k)$-orbit of $\Delta$.

 \item[(iii)] If $k'\in \Cc$ such that $\dist(k,k')=1$, then $\NN(k)$ intersects the $\Stab_\G(k)$-orbit of $k'$.
\end{itemize}
\end{Lemma}

\subsection{Algorithm A: enumerating elements of $\ol{\Cc}$ and $\ol{\Ic}$}\label{sec:quot:graph:diam}
We first describe an algorithm to enumerate elements of $\ol{\Cc}$ and $\ol{\Ic}$.
By Theorem~\ref{thm:area:TC:bound}, this algorithm provides us with a bound for $\Aa(\Hh/\G)$.
In what follows two elements of $\Cc$ (resp. $\Lc, \Ic$) are said to be {\em equivalent} if they belong to the same $\G$-orbit.

\begin{Remark}\label{rmk:equiv:dir}
To determine if two periodic directions $k$ and $k'$ are equivalent, one can proceed as follows: choose two matrices $A,A'$ such that $A(k)=A'(k')=\infty$, and the surfaces $A\cdot(M,\Sig)$ and $A'\cdot(M,\Sig)$ are normalized. Then $k$ and $k'$ are equivalent if and only if, up to the action of $\{U_t, \, t \in \R\}$ and Dehn twists in the horizontal cylinders, $A\cdot(M,\Sig)$ and $A'\cdot(M,\Sig)$ are represented by the same polygon in the plane.
\end{Remark}

\subsection*{Initialization:} Using the action of $\PSL(2,\R)$, we can assume that $M$ is horizontally periodic and normalized.
Set $\CC^0_0=\{\infty\}$.
Let $\CC_0^1$ be  a subset of $\NN(\infty)$ that satisfies
\begin{enumerate}
 \item[(a)] no element of $\CC_0^1$ is equivalent to $\infty$,

 \item[(b)] every element of $\NN(\infty)$ is equivalent to an element of $\CC_0^1$ or $\infty$,

 \item[(c)] no pair of elements of $\CC_0^1$ are  equivalent.
\end{enumerate}
The algorithm consists of exploring the graph $\GG$ from the vertex representing $\infty$ until we get a representative for every element of $\Cc/\G$.

\subsection*{Iteration:} Suppose now that we have two finite subsets $\CC_n^0$ and $\CC^1_n$ of $\Cc$ satisfying the following
\begin{enumerate}
  \item $\CC_n^0$ and $\CC_n^1$ are disjoint,
  \item no pair of directions in $\CC_n^0\sqcup \CC^1_n$ are equivalent,
\end{enumerate}
Assume that $\CC_n^1\neq \vide$. 
Set
$$
\hat{\NN}_{n+1}=\bigcup_{k \in \CC_n^1}\NN(k) \subset \Cc.
$$
Pick a subset $\hat{\NN}'_{n+1}$  of $\hat{\NN}_{n+1}$ such that
\begin{enumerate}
\item[(a)] no element of $\hat{\NN}'_{n+1}$ is equivalent to an element of $\CC_n^0\sqcup \CC_n^1$,
\item[(b)] every element of $\hat{\NN}_{n+1}$ is either equivalent to an element of $\CC_n^0\sqcup \CC_n^1$, or to a unique element of $\hat{\NN}'_{n+1}$.
\end{enumerate}
We now set
$$
\CC_{n+1}^0=\CC_n^0\sqcup\CC_n^1, \quad \CC_{n+1}^1:=\hat{\NN}'_{n+1}.
$$
Remark that we have $\CC^0_{n+1}=\{\infty\}\sqcup\CC^1_0\sqcup\dots\sqcup \CC^1_n$.
The algorithm stops when $\CC_n^1=\vide$.

\medskip

Consider now the graph $\ol{\GG}$. By definition, $\ol{\GG}$ has two types of vertices, let us denote by $\Vc$ the set of vertices of $\ol{\GG}$ representing the $\G$-orbits in $\Cc$, and by $\Wc$ the set of vertices representing the $\G$-orbits in $\Ic$.
Recall that by construction, every edge of $\ol{\GG}$ connects a vertex in $\Vc$ and a vertex in $\Wc$.
We denote by $\ol{\dist}$ the distance in $\ol{\GG}$.
Recall that Proposition~\ref{prop:G:dist:inters} and Proposition~\ref{prop:Vee:quot:finite} imply that $\ol{\GG}$ is a finite connected graph.

Let $v_\infty$ denote the vertex of $\ol{\GG}$ representing the $\G$-orbit of $\infty$.

\begin{Lemma}\label{lm:algo:A:dist:quot:G}
 Let $v$ be a vertex in $\Vc$. Then  $\ol{\dist}(v_\infty,v)=n$ if and only if $v$ represents the $\G$-orbit of a direction in $\CC^1_{n-1}$.
\end{Lemma}
\begin{proof}
It is clear from the construction that if  $v$ represents a direction in $\CC^1_0$ then $\ol{\dist}(v_\infty,v)=1$.
Conversely, if $\ol{\dist}(v_\infty,v)=1$ then $v$ represents the $\G$-orbit of a direction $k\in \Cc$ such that $\dist(\infty,k)=1$.
By Lemma~\ref{lm:neigh:1},  $k$ is equivalent to an element of $\NN(\infty)$.
Since $k$ is not equivalent to $\infty$, we can choose $k$ to be an element of $\CC^1_0$.

 Assume now that the lemma is true for $n\leq\ell$, and that  $\ol{\dist}(v_\infty,v)=\ell+1$. There exists $v'\in \Vc$ such that $\ol{\dist}(v_\infty,v')=\ell$ and $\ol{\dist}(v',v)=1$.
 By assumption, $v'$ represents the $\G$-orbit of a direction in $\CC^1_{\ell-1}$. Therefore, $v$ represents the $\G$-orbit of a direction  $k\in \hat{\NN}_{\ell}$.
 Note that $k$ cannot be equivalent to a direction in $\CC^0_{\ell-1}\sqcup\CC^1_{\ell-1}$, since otherwise we would have $\ol{\dist}(v_\infty,v)\leq \ell-1$ by the induction hypothesis.
 Thus $k$ must be equivalent to a direction in $\hat{\NN}'_{\ell}=\CC^1_\ell$. The lemma is then proved.
\end{proof}

As a direct consequence of Lemma~\ref{lm:algo:A:dist:quot:G}, we get
\begin{Proposition}\label{prop:algo:1:finite:steps}
Let $\dd_1$ be the maximal distance in $\ol{\GG}$ from $v_\infty$ to another vertex in $\Vc$.
Then the algorithm stops after $\dd_1$ iterations.
\end{Proposition}
\begin{proof}
It follows from Lemma~\ref{lm:algo:A:dist:quot:G} that $\CC^1_{\dd_1}=\vide$, thus the algorithm stops after $\dd_1$ iterations.
\end{proof}
Proposition~\ref{prop:algo:1:finite:steps} means that the number of iterations that have been performed when the algorithm stops equals to largest distance from $v_\infty$ to another vertex in $\Vc\subset \ol{\GG}$.
By construction, we have a bijection between $\CC^0_{\dd_1}$ and $\Vc$. Thus this algorithm allows us to get the complete list of elements of $\Vc$.
Recall that for any $k\in \Cc$, $\Ic^*(k)$ is the set of ideal triangles that form a reference domain $\DD(k)$ of $k$.
Set
$$
\Ic^*:=\bigcup_{k\in \CC^0_{\dd_1}}\Ic^*(k)
$$
It follows from Lemma~\ref{lm:ref:domain} that every $\G$-orbit in $\Ic$ has at least a representative in $\Ic^*$. Therefore, we can extract from $\Ic^*$ a subset $\Ic_0^*$ which is in bijection with $\ol{\Ic}\simeq \Wc$. By Theorem~\ref{thm:area:TC:bound}, the cardinality of $\Ic^*_0$ provides us with a bound for $\Aa(\Hh/\G)$.
Note also that the domain $\DD_0^*:=\cup_{k\in \Ic^*_0}\DD(k)$ has finite area and contains a fundamental domain for the action of $\G$ on $\Hh$.
Therefore, we can view $\DD^*_0$ as a ``coarse'' fundamental domain of $\G$.

\subsection{Algorithm B: finding a generating set of $\Gamma$}\label{sec:gen:set:Veech:gp}
We now present an algorithm to obtain a generating set of $\G$. In the literature, generating sets of a lattice in $\PSL(2,\R)$ are often obtained from a fundamental domain of the lattice. In this algorithm, we obtain a generating set of $\G$ without constructing explicitly a fundamental domain. In what follows we will use the same notation as in Section~\ref{sec:quot:graph:diam}.


\subsection*{Initialization:}
Let $g_\infty$ be a generator of the stabilizer of $\infty$ in $\G$.
Let $\Ic^*(\infty)$ be the set of ideal triangles in $\Ic$ which compose a reference domain $\DD(\infty)$ of  $\infty$.
We set
$$
\CC_0:=\{\infty\}, \quad \Jb_0:=\Ic^*(\infty), \quad \FF_0:=\{g_\infty\}.
$$
\subsection*{Iteration:} assume now that we have a finite subset $\CC_n$ of $\Cc$, and to every element of $\CC_n$ we have an associated finite subset $\Ic^*(k)$ of $\Ic$, and an element $g_k\in \G$ satisfying the followings: for any $k\in \CC_n$
\begin{itemize}
 \item[(i)] the elements of $\Ic^*(k)$ represent the ideal triangles which compose a reference domain for $k$ (in particular $k$ is a vertex of every ideal triangle in $\Ic^*(k)$),

\item[(ii)] if $k$ is equivalent to $\infty$  then $g_k(\infty)=k$ and $\Ic^*(k)=g_k(\Ic^*(\infty))$, otherwise $g_k$ is a generator of $\Stab_{\G}(k)$.
\end{itemize}
We set
$$
\Jb_n=\bigcup_{k\in \CC_n}\Ic^*(k) \subset \Ic \text{ and } \FF_n=\{g_k, \; k\in \CC_n\} \subset \G.
$$
Define $\CC_{n+1}$ to be the subset of $\Cc$ consisting of the vertices  the ideal triangles in $\Jb_n$.
Note that $\CC_n$ is a subset of $\CC_{n+1}$. We will associate to each $k\in \CC_{n+1}$ a finite subset $\Ic^*(k)$ of $\Ic$, and an element $g_k$ of $\G$ as follows: if  $k \in \CC_n$, we keep the same $\Ic^*(k)$ and $g_k$ provided by the previous step.
Let $k$ be an element of $\CC_{n+1}\setminus\CC_{n}$.  By definition, $k$ is a vertex of a triangle $\Delta\in \Jb_n$. We have two cases
\begin{itemize}
\item[$\bullet$] \underline{Case 1:} $k$ is equivalent to $\infty$. In this case, there is an element $g\in \G$ such that $g(\infty)=k$ and $g^{-1}(\Delta) \in \Ic^*(\infty)$.  We define $\Ic^*(k)=g(\Ic^*(\infty))$, and $g_k=g$.

\item[$\bullet$] \underline{Case 2:} $k$ is not equivalent to $\infty$.
  In this case, we choose a reference domain $\DD(k)$ such that $\Delta$ is one of the ideal triangles that make up $\DD(k)$. We then define $\Ic^*(k)$ to be the family of triangles that compose $\DD(k)$, and  $g_k$ a generator of $\Stab_\G(k)$.
\end{itemize}
Clearly, $\CC_{n+1}$ and the mappings $k \mapsto \Ic^*(k)$, and $k\mapsto g_k$  satisfy the conditions (i) and (ii) above.

\begin{Lemma}\label{lm:algo:B:basic:prop}
For any $n \in \N$ and any $k \in \CC_n$, we have
\begin{itemize}
 \item[a)] The subgroup generated by $\FF_n$ contains the stabilizer of $k$ in $\G$.

 \item[b)] As subsets of $\GG$, $\CC_{n+1}$ is contained in the $1$-neighborhood of $\CC_n$.

 \item[c)] If $k \in \CC_{n+1}$, then the distance from $\infty$ to $k$  in $\GG$ is at most $n$.
\end{itemize}
\end{Lemma}
\begin{proof}
For a), we only need to consider the case $k$ is equivalent to $\infty$. But in this case, $g_k\cdot g_\infty\cdot g_k^{-1}$ is a generator of $\Stab_\G(k)$.
For b), observe that $\Jb_n$ is contained in the $\frac{1}{2}$-neighborhood of $\CC_n$, and $\CC_{n+1}$ is contained in the $\frac{1}{2}$-neighborhood of $\Jb_n$. Finally, c) is an immediate consequence of b).
\end{proof}

Let $\G_n$ denote the subgroup of $\G$ that is generated by the elements of $\FF_n$.

\begin{Lemma}\label{lem:gen:set:dist:n}
  Let $A$ be an element of $\G$ such that $\dist(\infty,A(\infty)) \leq n$, where $\dist$ is the distance on the graph $\GG$. Then $A \in \G_n$.
\end{Lemma}
\begin{proof}
Let $\alpha$ be a path of minimal length from $\infty$ to $A(\infty)$ in $\GG$. Let $m=\leng(\alpha) \leq n$. Then $\alpha$ must contain $m+1$ vertices in $\Cc$.
Let us label those vertices by $k_0,k_1,\dots,k_m$, where $k_0=\infty, k_m=A(\infty)$ and $\dist(\infty,k_i)=i$.

For any $k\in \CC_n$, define
$$
\NN(k)=\{k' \in \Cc, \, k'\neq k, k' \text{ is a vertex of some triangle in } \Ic^*(k)\}.
$$
Since $\dist(k_0,k_1)=1$, by Lemma~\ref{lm:neigh:1}, there is an element $B_0\in \Stab_{\G}(k_0)$ such that $k'_1:=B_0(k_1)\in \NN(k_0)$.
Note that $k'_1\in \CC_1$ and $B_0\in \G_0$.

Let $k'_2:=B_0(k_2)$. Since $\dist(k'_1,k'_2)=\dist(k_1,k_2)=1$, there is an element $B_1\in \Stab_{\G}(k'_1)$ such that $k''_2:=B_1(k'_2)=B_1\circ B_0(k_2) \in \NN(k'_1)$. In particular, we have $k''_2\in \CC_2$, and $B_1\in \G_1$ by Lemma~\ref{lm:algo:B:basic:prop}.

By induction, we can find a sequence $(B_0,B_1,\dots,B_{m-1})$ of elements of $\G$ such that $B_i \in \G_i$, and $B_{m-1}\circ\dots\circ B_0(k_m)=k^{(m)}_m \in \CC_{m}$.
Since $k_m$ is equivalent to $\infty$, by construction, there is an element $B_{m} \in \FF_{m}$ such that $B_{m}(k^{(m)}_m)=\infty$. Hence
$$
B_m\circ B_{m-1}\circ\dots\circ B_0\circ A(\infty)=\infty.
$$
which means that there exists $B\in \Stab_{\G}(\infty)=\G_0$ such that
$$
A =B_0^{-1}\circ\dots\circ B^{-1}_{m}\circ B \in \G_{m}.
$$
\end{proof}

Let $v_\infty$ be the vertex of $\ol{\GG}$ that represents the $\G$-orbit of $\infty$ in $\Cc$.
Recall that we have defined $\dd_1$ to be the maximal distance in $\ol{\GG}$ from $v_\infty$ to another vertex in $\Vc$ (that is the set of $\G$-orbits in $\Cc$). Note that $\dd_1$  can be computed by  Algorithm A (c.f. Proposition~\ref{prop:algo:1:finite:steps}).

\begin{Proposition}\label{prop:generate:Gam}
 We have $\G_{2d_1+1}=\G$.
\end{Proposition}
\begin{proof}
Let $A$ be an element of $\G$. Let $m:=\dist(\infty,A(\infty))$.  We will prove that $A \in \G_{2d_1+1}$ by induction on $m$.
For $m\leq 2d_1+1$, this follows from Lemma~\ref{lem:gen:set:dist:n}. Thus let us suppose that $m>2d_1+1$, and that the statement is true for any $A$ such that $\dist(\infty,A(\infty))< m$.

Let $\alpha$ be any path of minimal length in $\GG$ from $\infty$ to $A(\infty)$. This path contains $m+1$ vertices in $\Cc$ that are labeled by $k_0,\dots,k_m$, where $k_0=\infty, k_m=A(\infty)$, and $\dist(k_0,k_i)=i$. Consider the vertex $k_{m-d_1-1}$. Since the vertex of $\ol{\GG}$ that represents the $\G$-orbit of $k_{m-d_1-1}$ is of distance at most $d_1$ from $v_\infty$, there is a vertex $k\in \Cc$ in the  $\G$-orbit of $\infty$ such that $\dist(k_{m-d_1-1},k) \leq d_1$. Consequently, $\dist(\infty,k) \leq (m-d_1-1)+d_1=m-1$, and $\dist(k,k_m)\leq d_1+d_1+1=2d_1+1$.

By assumption, there is an element $A'\in \G$ such that $A'(\infty)=k$. By the induction hypothesis, $A'\in \G_{2d_1+1}$. Consider $k':={A'}^{-1}(A(\infty))$.
Now, since
$$
\dist(\infty,k')=\dist(k,A(\infty))\leq 2d_1+1
$$
the matrix ${A'}^{-1}\cdot A$ belongs to $\G_{2d_1+1}$ by Lemma~\ref{lem:gen:set:dist:n}. Thus $A\in \G_{2d_1+1}$, and the proposition is proved.
\end{proof}

Proposition~\ref{prop:generate:Gam} implies that we obtain a generating set for the Veech group of $(M,\Sig)$ after $2d_1+1$ iterations of Algorithm B.


\end{document}